\documentclass[a4paper,12pt]{article}

\usepackage{mathtools}
\usepackage[latin1]{inputenc}
\usepackage[all]{xy}
\usepackage{float}
\usepackage{amsmath}
\usepackage{amssymb}
\usepackage{amsfonts}
\usepackage{amsthm}
\usepackage{enumerate}
\usepackage{epstopdf}
\usepackage{color,fancybox,graphicx}
\usepackage{url}
\usepackage{psfrag}
\usepackage{color}
\usepackage{graphicx,psfrag}
\usepackage{epsfig}
\renewcommand{\leq}{\leqslant}
\renewcommand{\ge}{\geqslant}
\renewcommand{\geq}{\geqslant}

\newcommand{\cadlag}{c\`adl\`ag }
\newcommand{\E}{\mathbb{E}}

\newcommand{\N}{\mathbb{N}}
\newcommand{\M}{\mathbb{M}}
\newcommand{\R}{\mathbb{R}}

\renewcommand{\P}{\mathbb{P}}
\renewcommand{\L}{\mathbb{L}}

\newcommand{\calL}{\mathcal{L}}

\newcommand{\calM}{\mathcal{M}}
\newcommand{\calD}{\mathcal{D}}
\newcommand{\calF}{\mathcal{F}}

\newcommand{\calG}{\mathcal{G}}

\newcommand{\eps}{\varepsilon}
\newcommand{\ph}{\varphi}

\renewcommand{\d}{ {\, d}}

\newcommand{\eqdef}{:=}

\newcommand{\dps}{\displaystyle}

\newcommand{\abs}[1]{\left | #1\right |}
\newcommand{\set}[1]{\left\{#1\right\}}

\newcommand{\p}[1]{ \left(#1\right) }
\renewcommand{\b}[1]{ \left [#1\right ] }

\newcommand{\norm}[1]{\left\Vert#1\right\Vert}
\newcommand{\bracket}[1]{\left \langle #1\right \rangle}

\theoremstyle{plain}
\newtheorem{The}{Theorem}[section]
\newtheorem{Lem}[The]{Lemma}
\newtheorem{Pro}[The]{Proposition}

\newtheorem{Cor}[The]{Corollary}
\newtheorem*{Ass}{Assumption}

\newtheorem{Def}[The]{Definition}

\numberwithin{equation}{section}

\theoremstyle{definition}
\newtheorem{Rem}[The]{Remark}

\numberwithin{equation}{section}

\newcommand{\un}{{\mathbf{1}}}

\newcommand{\V}{\mathbb V}

\def\R{\mathbb{R}}
\def\P{\mathbb{P}}
\def\E{\mathbb{E}}

\begin{document}

\begin{center}
{\sc \Large A Central Limit Theorem
 %\\
%\vspace{0.2cm}
for Fleming-Viot Particle Systems with Soft Killing\footnote{This work was partially supported by the French Agence Nationale de la Recherche, under grant ANR-14-CE23-0012, and by the European Research Council under the European Union's Seventh Framework Programme (FP/2007-2013) / ERC Grant Agreement number 614492.}}
\vspace{0.5cm}

\end{center}

{\bf Fr\'ed\'eric C\'erou\footnote{Corresponding author.}}\\
{\it INRIA Rennes \& IRMAR, France }\\
\textsf{frederic.cerou@inria.fr}
\bigskip

{\bf Bernard Delyon}\\
{\it Universit\'e Rennes 1 \& IRMAR, France }\\
\textsf{bernard.delyon@univ-rennes1.fr}
\bigskip

{\bf Arnaud Guyader}\\
{\it Universit\'e Pierre et Marie Curie \& CERMICS, France }\\
\textsf{arnaud.guyader@upmc.fr}
\bigskip

{\bf Mathias Rousset}\\
{\it INRIA Rennes \& IRMAR, France }\\
\textsf{mathias.rousset@inria.fr}
\bigskip

\medskip

\begin{abstract}
\noindent {\rm The distribution of a Markov process with killing, conditioned to be still alive at a given time, can be approximated by a Fleming-Viot type particle system. In such a system, each particle is simulated independently according to the law of the underlying Markov process, and branches onto another particle at each killing time. The consistency of this method in the large population limit was the subject of several recent articles. In the present paper, we go one step forward and prove a central limit theorem for the law of the Fleming-Viot particle system at a given time under two conditions: a ``soft killing'' assumption and  a boundedness condition involving the ``carr\'e du champ'' operator of the underlying Markov process. 
\medskip

\medskip

}

\end{abstract}

\tableofcontents
%----------------------------------------------------------------------
%----------------------------------------------------------------------
\section{Introduction}\label{intro}
%----------------------------------------------------------------------
%----------------------------------------------------------------------

\paragraph{Context and notation}

Let $X=(X_t)_{t\geq 0}$ be a Markov process evolving in $F\cup\{\partial\}$, where $\partial\notin F$ is absorbing and $F$ is the state space. Specifically, $X$ evolves in $F$ until it reaches $\partial$ and then remains trapped in this cemetery point forever. The initial distribution is denoted $\eta_0$, with the assumption that $\eta_0(\{\partial\})=0$.  Let us also denote $\tau_\partial$ the killing time of this process, meaning that
$$\tau_\partial\eqdef  \inf\{t\geq 0, X_t=\partial\}.$$
Then, given a deterministic final time $T>0$, we are interested both in the distribution of $X_T$ given that it has still not been killed at time $T$, denoted
$$\eta_T\eqdef  {\cal L}(X_T|\tau_\partial>T)={\cal L}(X_T|X_T\neq\partial),$$
and in the probability of this event, that is
$$p_T\eqdef  \P(\tau_\partial>T)=\P(X_T\neq\partial).$$
A crude Monte Carlo method in order to approximate these quantities consists in:
\begin{itemize}
\item simulating $N$ i.i.d.~random variables, also called particles in the present work, 
$$X_0^1,\dots,X_0^N\ \overset{\rm i.i.d.}{\sim}\ \eta_0,$$ 
\item letting them evolve independently according to the dynamic of the underlying process $X$, 
\item and eventually considering the estimators
$$\hat{\eta}_T^N\eqdef  \frac{\sum_{i=1}^N \un_{X_T^i\in F}\ \delta_{X_T^i}}{\sum_{i=1}^N\un_{X_T^i\in F}}\hspace{1cm}\mbox{and}\hspace{1cm}\hat{p}_T^N\eqdef  \frac{\sum_{i=1}^N\un_{X_T^i\in F}}{N},$$
with the convention that $0/0=0$.
\end{itemize}
It is readily seen that if $\partial$ is attractive, these estimators will not be relevant for large $T$ since we face a rare event estimation problem. A possible way to tackle this issue is to approximate the quantities at stake through a Fleming-Viot type particle system \cite{bhim96,v14}. Under assumptions that will be detailed below, the following algorithm is (almost surely) well defined:
\begin{Def}[Fleming-Viot particle system]\label{peicj}
The Fleming-Viot particle system $(X^1_t, \cdots, X^N_t)_{t \in [0,T]}$ is the Markov process with state space $F^N$ defined by:
\begin{itemize}
\item Initialization: consider $N$ i.i.d.~particles
$$X_0^1,\dots,X_0^N\ \overset{\rm i.i.d.}{\sim}\ \eta_0,$$
\item Evolution and absorption (or killing): each particle evolves independently according to the law of the underlying Markov process $X$ until one of them is absorbed in the cemetery point $\partial$,
\item Branching (or rebirth): the absorbed particle is taken from $\partial$, and is given instantaneously the state of one of the $(N-1)$ other particles (randomly uniformly chosen),
\item and so on until final time $T$.
\end{itemize}
\end{Def}
Finally, consider the estimators
$$\eta_T^N\eqdef  \frac{1}{N}\sum_{i=1}^N\delta_{X_T^i}\hspace{1cm}\mbox{and}\hspace{1cm}p_T^N\eqdef  \left(1-\tfrac{1}{N}\right)^{B_T},$$
where $B_T=B_T^N$ is the total number of branchings of the particle system until final time $T$.\medskip

Under very general assumptions, Villemonais \cite{v14} proves among other things that $p_T^N$ converges in probability to $p_T$ when $N$ goes to infinity, and that $\eta_T^N$ converges in law to $\eta_T$.\medskip

The purpose of this paper is to go one step further and to establish, under more restrictive assumptions, central limit results for $\eta_T^N$ and $p_T^N$. It turns out that both quantities can be handled by considering the unnormalized measure
$$\gamma_T\eqdef  p_T \eta_T,$$
and its empirical approximation 
$$\gamma_T^N\eqdef  p_T^N \eta_T^N.$$
Note that for any $t\in[0,T]$, one has $p_t=\gamma_t(\un_F)=\P(\tau_\partial>t),$
so that $p_0=1$.\medskip

Let us also introduce the semi-group operator $Q^h$ defined for any bounded measurable function $\ph:F\cup\{\partial\}\to\R$, for any $x\in F\cup\{\partial\}$ and for any $h\geq0$ by
$$Q^h\ph(x) \eqdef \E_x[\ph(X_h)].$$
By convention, if $\ph$ is defined on $F$ then we extend it on $F\cup\{\partial\}$ by setting $\ph(\partial)=0$, in which case we have $Q^h\ph(\partial)=0$  for all $h\geq 0$. \medskip

Furthermore, for any probability distribution $\mu$ on $F$ and any bounded measurable function $\ph:F\to\R$, the standard notation $\V_\mu(\ph)$ stands for the variance of the random variable $\ph(Y)$ when $Y$ is distributed according to $\mu$, i.e.
$$\V_\mu(\ph)\eqdef  \V(\ph(Y))=\E[\ph(Y)^2]-\E[\ph(Y)]^2=\mu(\ph^2)-\mu(\ph)^2.$$

\paragraph{Main assumptions}

The ``soft killing'' assumption~(SK) that will be specified in Section~\ref{AMEKCN} requires that there exists a bounded function $\lambda:F\to\R^+$ defining the intensity of the absorption in $\partial$ from point $x$, meaning that, for all $x\in F$, 
\begin{equation}\label{apmicjk}
\lambda(x)\eqdef \lim_{h\downarrow 0}\frac{\P(X_{t+h}=\partial | X_t=x)}{h}.
\end{equation}
Said differently, the mapping $t\mapsto p_t=\P(\tau_\partial>t)$ is differentiable on $[0,T]$ with derivative
\begin{equation}\label{apmicj}
p'_t \eqdef  \frac{\d}{\d t} p_t =\frac{\d}{\d t}\gamma_t(\un_F)=-\E\left[\lambda(X_t)\un_{t< \tau_\partial}\right]=-\gamma_t(\lambda\un_F)=-\gamma_t(\lambda),
\end{equation}
since by convention $\lambda(\partial)=0$.

\medskip

We will also need a  ``carr\'e du champ'' assumption~(CC) that will be specified in Section~\ref{AMEKCN}. 
% A typical sufficient condition for a bounded measurable function $\ph$ to satisfy Assumption~(CC) is that the time derivative of the average conditional variance $$
% 0 \leq - \frac{\d}{\d t} \E \b{ \V( \ph(X_T) | X_t) } \leq c(\ph) < +\infty$$
% is almost surely bounded by a constant $c(\ph)$ independent of time and of the initial law $\calL(X_0)$; in the above the conditional variance is the variance associated to the conditional expectation $\E \b{ \, | X_t}$. 
This assumption defines a set of sufficiently regular bounded test functions $\ph$, and is related to the regularity of the underlying Markov process. It is satisfied for instance by:
\begin{itemize}
\item any bounded function $\ph$ for Piecewise Deterministic Markov Processes with bounded jump intensity (see Section~\ref{sec:examples});
\item  any smooth function $\ph$ for regular enough diffusions (see Section~\ref{sec:examples}).
\end{itemize}

\paragraph{Result}

The main result of this paper says that, under Assumption~(SK) and for any $\ph$ in the $\norm{ \cdot }_\infty$-closure of the set of functions satisfying Assumption~(CC), we have
$$\sqrt{N}\left(\gamma_T^N(\ph)-\gamma_T(\ph)\right)\xrightarrow[N\to\infty]{\cal D}{\cal N}(0,\sigma_T^2(\ph)),$$
where
\begin{equation}\label{eq:var_00}
\sigma_T^2(\ph) \eqdef p^2_T \V_{\eta_T}(\ph) - p_T^2\ln(p_T) \, \eta_T(\ph)^2 - 2\int_0^T \V_{\eta_{t}}(Q^{T-t}(\ph)) p_t p'_t  \d t.
\end{equation}

Then it suffices to take $\ph=\un_F$ to get a central limit result for $p_T^N=\gamma_T^N\left(\un_F\right)$, and to consider the decomposition 
$$\eta_T^N\left(\ph\right)-\eta_T\left(\ph\right)=\frac{\gamma_T(\un_F)}{\gamma_T^N(\un_F)} \frac{1}{p_T} \gamma_T^N(\ph-\eta_T\left(\ph\right))$$
to deduce a  central limit result for $\eta_T^N\left(\ph\right)$.\medskip

Before proceeding, let us mention that in Section~$3$ of~\cite{dm00}, Del Moral and Miclo also propose a central limit theorem for a class of interacting particle systems. However, there are some significant differences with the Fleming-Viot algorithm of Definition \ref{peicj}:
\begin{enumerate}[(i)]
 \item They consider unnormalized semi-groups with growing (instead of decreasing) probability mass. Therefore the killing in our context is replaced by a splitting in their algorithm. Namely, when a branching event occurs, the particle splits into two new particles and a uniformly chosen other particle is killed in order to control the population size. Their splitting is also assumed ``soft'', with bounded intensity.
 \item In their context, the probability $\P(\tau_\partial > T)$ is estimated by the analog of $\exp(- \int_0^T \eta^N_t(\lambda) \d t)$. The latter is an exponential kind of compensator computed from the Dol\' eans-Dade exponential of the branching counting process.
 \item As a consequence, their asymptotic variance is substantially different from ours, though both have a similar structure. 
\end{enumerate}
Therefore, even if the present work shares some features with \cite{dm00}, these differences modify many crucial aspects of the proof, including the final variance formula and the calculation of the quadratic variation of martingales.\medskip

%More importantly, a possible use of our result is the analysis of the so-called Adaptive Multilevel Splitting algorithm, see for example \cite{cg2,ghm} for a general version and \cite{cglp} for its use in molecular simulation. The application of the present results to this context is not completely straightforward and will be the subject of a future paper. Nonetheless, let us already stress that in this specific context, the intensity of killing $\lambda$ \emph{is not explicitly known nor estimated}: this fact explains why variant~(ii) mentioned above can definitely not be considered.\medskip

As far as we know, there is still no CLT result in the case of ``hard killing'' (see discussion below), and this case seems more challenging. Nevertheless, there is a cluster of papers considering the hard killing case where $X_t$ is a diffusion process in a bounded domain of $\R^d$ killed when it hits the domain boundary. Among other questions, the convergence of the empirical measures as $N$ goes to infinity is addressed in \cite{BBF12,GK04,lobus} (see also references therein). This case is also included in the general convergence results of \cite{v14}.\medskip 
   
\paragraph{Examples and counter-examples of soft killing}

Before going into more detail on the precise statements of our results, we expose a few elementary examples to explain what we mean by ``soft killing'' in order to circumscribe the scope of this paper.\medskip

Our first example is a classical ruin problem in non-life insurance (see for example \cite{aa10,m04}). In this context, $S_t$ represents the insurance portfolio at time $t$, defined by
$$S_t=s_0+ct-\sum_{i=1}^{N_t}Y_i,$$
where $s_0$ is the initial value of the portfolio, $c$ is the premium rate, $(N_t)_{t\geq0}$ is an homogeneous Poisson process with intensity $\theta$, called the claim number process, and $Y_1,Y_2,\ldots$ is a sequence of i.i.d.~non-negative random variables called the claim sizes or claim severities.\medskip

Here a killing corresponds to the ruin of the insurance company and the killing time is thus defined as $\tau_\partial=\inf\{t\geq0,\ S_t<0\}$. In this case, $F=[0,\infty)$ and the Markov process with killing $(X_t)_{t\geq 0}$ with values in $F\cup\{\partial\}$ is defined by $X_t=S_t$ if $t<\tau_\partial$ and $X_t=\partial$ if $t\geq\tau_\partial$. A standard result (see for example \cite{aa10}, Chapter IV, Corollary 1.4) says that, for any $x_0\geq 0$,
$$\tau_\partial<+\infty\ \ a.s.\ \Longleftrightarrow\ c\leq \theta \E[Y_1],$$
in which case the ruin is an attractive state. In this situation, denoting $F_Y(y)=\P(Y_1\leq y)$ the cdf of $Y_1$, we have for any $x\geq0$,
$$\lambda(x)\eqdef  \lim_{h\downarrow 0}\frac{\P(X_{t+h}=\partial | X_t=x)}{h}=\lim_{h\downarrow 0}\frac{\P(S_{t+h}<0 | S_t=x)}{h}=\theta(1-F_Y(x)),$$
so that $\|\lambda\|_\infty<\infty$. As we will see in Section \ref{sec:examples}, this is an example of Piecewise Deterministic Markov Process (PDMP for short) to which our results will apply.\medskip

By contrast, our second example illustrates the notion of ``hard killing''. It deals with another family of PDMP, called Additive-Increase Multiplicative-Decrease Markov processes. From an application viewpoint, these processes have connections with methods like TCP/IP  (Transmission Control Protocol/Internet Protocol) to control congestion in communication networks (see for example \cite{dgr02}).\medskip

This time, $(S_t)_{t\geq 0}$ is defined by $S_0=s_0\geq 0$ and, for all $t>0,$
$$S_t=\left\{\begin{array}{ll}
S_{T_n}+t-T_n&\mbox{if}\ T_n\leq t<T_{n+1}\\
Q_{n+1}(S_{T_n}+T_{n+1}-T_n)&\mbox{if}\ t=T_{n+1}
\end{array}\right.$$
where $T_0=0$, $(T_n)_{n\geq 1}$ are the arrival times of a homogeneous Poisson process, and $Q_1,Q_2,\ldots$ are i.i.d.~non-negative random variables with values in $[0,1)$.\medskip

In this context, a killing happens when $S_t$ reaches a fixed given value $s_{\max}$, which means that $\tau_\partial=\inf\{t>0,\ S_t=s_{\max}\}$. Hence, $F=[0,s_{\max})$ and the Markov process with killing $(X_t)_{t\geq 0}$ with values in $F\cup\{\partial\}$ is defined by $X_t=S_t$ if $t<\tau_\partial$ and $X_t=\partial$ if $t\geq\tau_\partial$. For any $x$ in $[0,s_{\max})$, we have that
$$\lim_{h\downarrow 0}\frac{\P(X_{t+h}=\partial | X_t=x)}{h}=\lim_{h\downarrow 0}\frac{\P(S_{t+h}\geq s_{\max} | X_t=x)}{h}=0.$$
Clearly,  (\ref{apmicjk}) and (\ref{apmicj}) cannot be satisfied simultaneously. As we will see in Section \ref{AMEKCN}, the central limit theorems of the present article will not apply in this kind of situation.\medskip

Note that in the framework of general stochastic calculus, the difference between ``hard killing'' and ``soft killing'' can be interpreted through the dichotomy between \emph{predictable} stopping times and \emph{totally inaccessible} stopping times. We refer the interested reader to Chapter I, Section~$2$, of~\cite{js03} for definitions, together with Theorem~$2.22$ for results on this topic.\medskip 

%A ENLEVER ET A METTRE DANS LA SECTION CALCUL STO Let us recall that a predictable stopping is a stopping time $\tau_\infty = \lim_{n \to + \infty} \tau_n $ that is approached from the left strictly (``is anticipated'') $\tau_n < \tau_\infty$ and encompass the idea of ``hard'' killing time; while a totally inaccessible stopping time is a stopping time that never equate any predictable stopping time (with probability one), and encompass the idea of ``soft`` killing time. A ENLEVER

Finally, these notions of hard and soft killing are also related to hard and soft obstacles. %The following frameworks are very standard, and were probably among the very first ones where the Fleming-Viot type particle algorithm was applied (REFERENCE HERE). 
To keep it simple, consider a physics particle whose motion is for example given by a diffusion process $X$. The particle can evolve in free space or encounter an obstacle. In free space, the particle's trajectory is given by $X$ until it reaches an obstacle. If the obstacle is ``hard'' then the particle is killed as soon as it touches it. If the obstacle is ``soft'', its dynamics is unchanged, but it is killed with some given intensity  $\lambda>0$ as long as it stays in the obstacle. An interesting subcase of the ``hard'' case is the particle in a box: the free space is the interior of a compact set, and the obstacle is all the rest. In the ``hard'' case, we cannot define a killing intensity for the same reason as in the TCP/IP case above.\medskip
 
%If the distribution of the particle at time $t$ given it is not killed converges when $t\rightarrow +\infty$, then the limiting distribution is called {\it quasi-stationary}. 
Our central limit theorem below will apply only in the ``soft'' case, as our proof relies on the existence of a killing intensity $\lambda$. Even if expression \eqref{eq:var_00} for the asymptotic variance can make sense without this killing intensity, whether or not there is a central limit theorem for ``hard'' obstacles is still an open question.

%----------------------------------------------------------------------
%----------------------------------------------------------------------
\section{Main result}\label{AMEKCN}
%----------------------------------------------------------------------
%----------------------------------------------------------------------

%----------------------------------------------------------------------
\subsection{Notation and assumptions}\label{zcijo}
%----------------------------------------------------------------------

Throughout the paper, for the sake of simplicity, the state space $F$ is assumed Polish, and the underlying Markov process $X=(X_t)_{t\geq 0}$ is assumed to be c\`adl\`ag, although the specific topology will not play any role (see Section~\ref{sec:def_ips} for comments on a setting without topology). Besides, the process is assumed time-homogeneous with associated  Markov semi-group $(Q^t)_{t \geq 0}$. Specifically, $X=(X_t)_{t\geq 0}$ evolves in $F\cup\{\partial\}$, where $\partial\notin F$ is absorbing, and with absorption time $\tau_\partial$.
% $$\tau_\partial\eqdef  \inf\{t\geq 0, X_t=\partial\}.$$
More importantly, we assume that this Markov process with killing is Markov with respect to the minimal \emph{right-continuous} filtration it generates. More comments on this topic are provided in Section~\ref{sec:def_ips}.
%*** Mathias: j'ai enleve strong Markov, je pense qu'on en a pas besoin, a verifier tout de meme. En revanche, l'hypothese Markov pour la filtration continue a droite permet d'avoir des $\mathbb{M}$ cadalg, ce qui est pratique pour le calcul sto utilise *** %$(\calF_t)_{t\geq 0}$. %(voir si on garde ça : A possible rigorous setting is the following: the data is a measurable map $x \mapsto \mu_x$ which associates to each initial condition $x \in F$, a distribution $\mu_x$ in the Borel space of c\`adl\`ag paths (the Skorokhod space) in $F \cup \set{\partial}$. It is assumed that $\mu_x$ satisfies the strong Markov property with respect to the minimal canonical right-continuous filtration of the Skorokhod space, that is to say, at time $t$, the intersection of $\sigma(x_{s},s \leq t+h)$ when $h \downarrow 0$).
Our first assumption allows us to define an intensity of absorption for the process and was already heuristically discussed in the introduction.

\begin{Ass}[{\bf SK}](Soft Killing assumption)\label{ass:lambdma} There is a bounded measurable function $\lambda: F \to \R^+$ such that, for any initial distribution of $X_0$, the process
 \[
t\mapsto  \un_{X_t=\partial} - \int_0^{t \wedge \tau_\partial} \lambda(X_s) \d s
 \]
is a martingale with respect to the minimal right-continuous filtration generated by $X$.
 \end{Ass}
 
%In all what remains, we adopt the convention that $\lambda(\partial)=0$. 
%In the next assumption, $C_b(F)$ stands for the set of continuous and bounded functions defined on $F$ and with values in $\R$. 
Let us also recall that by the Markov property, the process $t \mapsto Q^{T-t}(\ph)(X_t)$ is a martingale with a c\`adl\`ag version (see~\cite{ry99} Chapter~II, or the proof of Lemma~\ref{lem:mart0}).

\begin{Ass}[{\bf CC}](Carr\'e du Champ assumption)\label{ass:gamma} A bounded measurable function $\ph:F \to \R$ is said to satisfy Assumption~(CC) if there is a measurable function $(t,x) \in \R^+ \times F \mapsto \Gamma_{t}(\ph)(x)$ satisfying for any $T > 0$
  $$ \dps \int_0^T \norm{\Gamma_{t} (\ph) }_\infty \d t < +\infty,$$ 
  and such that, for any initial distribution of $X_0$, the process
  $$
 t \in [ 0,T ] \mapsto \p{  Q^{T-t}(\ph)(X_t) }^2  - \int_0^{t \wedge \tau_\partial} \Gamma_{T-s}(\ph)(X_s) \d s
  $$
is a martingale for the minimal right-continuous filtration generated by $X$. 
\end{Ass}

We will comment on these assumptions and provide some specific examples in Section \ref{pmazioch}.

%----------------------------------------------------------------------
\subsection{Main result}\label{sec:main}
%----------------------------------------------------------------------

We keep the notation of Section~\ref{intro}. In particular, $(X^1_t, \ldots, X^N_t)_{t \geq 0}$ denotes the c\`adl\`ag Fleming-Viot particle system. The filtration $(\calF_t)_{t \geq 0}$ is the minimal right-continuous filtration generated by this particle system (see also Section~\ref{sec:def_ips} for details and comments).\medskip

For any $n \in \set{1, \ldots , N}$ and any $k \geq 1$, we denote by $\tau_{n,k}$
the $k$-th branching time of particle $n$, with the convention $\tau_{n,0} =0$. Moreover, for any $j \geq 1$, we denote by $\tau_{j}$ the $j$-th branching time of the whole system of particles. Accordingly, the processes
$$
B^{n}_t\eqdef   \sum_{k\geq 1} \un_{\tau_{n,k}\leq t}
$$
and  
$$
B_t\eqdef   \sum_{n=1}^N B^{n}_t= \sum_{j\geq 1} \un_{\tau_{j}\leq t}
$$
are c\`adl\`ag counting processes that correspond respectively to the number of branchings of particle $n$ before time $t$, and to the total number of branchings of the whole particle system before time $t$.\medskip

As mentioned before, we can then define the empirical measure associated to the particle system as 
$ \eta^N_t\eqdef   \frac{1}{N} \sum_{n=1}^N \delta_{X_t^n},$
while the estimate of the probability that the process is still not killed at time $t$ is denoted 
$p^N_t\eqdef   (1-\tfrac{1}{N})^{B_t},$
and the unnormalized empirical measure is defined as $\gamma^N_t\eqdef   p^{N}_t \eta^N_t$.\medskip

As will be recalled in Proposition \ref{pro:estimate} and already noticed by Villemonais in \cite{v14}, for any bounded $\ph$, their large $N$ limits are respectively 
%$$\left\{\begin{array}{rl}
$\eta_t(\varphi)\eqdef\E [\varphi(X_t)|X_t\neq\partial]$,
$p_t\eqdef \P(X_t\neq\partial)$, and
%\\
$\gamma_t(\varphi) \eqdef \E[\varphi(X_t)\un_{X_t\neq\partial}]$.
%\end{array}\right.$$
We clearly have $\eta_t(\varphi)=\gamma_t(\varphi)/\gamma_t(\un_F)=\gamma_t(\varphi)/p_t$ and $\gamma_t(\varphi)=\eta_0(Q^t \varphi)$.\medskip 

By Assumption~(SK), we get
$$p_t=\gamma_t(\un_F)=\P(X_t\neq\partial)=1-\E\left[\int_0^{t} \lambda(X_s)\un_{s\leq \tau_\partial} \d s\right].$$
Since $\lambda$ is assumed bounded, this ensures that the mapping $t\mapsto p_t$ is differentiable with derivative
$$p'_t = \frac{\d}{\d t} p_t =\frac{\d}{\d t}\gamma_t(\un_F)=-\E\left[\lambda(X_t)\un_{t\leq \tau_\partial}\right]=-\gamma_t(\lambda\un_F)=-\gamma_t(\lambda).$$
We can now expose the main result of the paper.

\begin{The}\label{gamma}
Denote by ${\cal D}$ the set of bounded measurable functions satisfying Assumption~(CC) and by $\overline{\cal D}$ its closure with respect to the uniform norm $\norm{\cdot}_\infty$. Then, under Assumption~(SK), for any $\ph$ in $\overline{\cal D}$, one has
$$\sqrt{N}\left(\gamma_T^N(\ph)-\gamma_T(\ph)\right)\xrightarrow[N\to\infty]{\cal D}{\cal N}(0,\sigma_T^2(\ph)),$$
where
$$\sigma_T^2(\ph) = p^2_T \V_{\eta_T}(\ph) - p_T^2\ln(p_T) \, \eta_T(\ph)^2 - 2\int_0^T \V_{\eta_{t}}(Q^{T-t}(\ph)) p_t p'_t  \d t.$$
\end{The}

As we will see, we prove in fact a somehow stronger result. Indeed, Corollary \ref{lazicj} explains that for any $\ph$ in ${\cal D}$, the martingale $(Z_t^N)_{0\leq t\leq T}$ defined by
 $$Z_t^N:=\sqrt{N}\left(\gamma_t^N(Q^{T-t}(\ph))-\gamma_0(Q^{T}(\ph))\right)$$
converges in law towards a Gaussian process $(Z_t)_{t \in [0,T]}$ with independent increments, initial distribution ${\cal N}(0,\V_{\eta_0}(Q^{T}(\ph)))$, and variance function 
\begin{equation}\label{mpoecj}
\sigma_t^2(\ph)=\V_{\eta_0}(Q^{T-t}\ph)+  \int_0^{t} \b{ \eta_{s} \p{\Gamma_{T-s}(\ph)} + \V_{\eta_{s}}(Q^{T-s}(\ph))     \eta_s(\lambda) }   p_{s}^2 \d s.
\end{equation} 
In particular, thanks to a density argument, Corollary \ref{corbis} ensures that for any $\ph$ in $\overline{\cal D}$, one has 
$$\sqrt{N}\left(\gamma_T^N(\ph)-\gamma_T(\ph)\right)\xrightarrow[N\to\infty]{\cal D}{\cal N}(0,\sigma_T^2(\ph)),$$
and we eventually explain why expression (\ref{mpoecj}) for $\sigma_T^2(\ph)$ indeed coincides with the one given in Theorem \ref{gamma}.\medskip

However, coming back to Theorem \ref{gamma}, the CLT for $\eta_T^N$ is then a straightforward application of this result by considering the decomposition 
$$\eta_T^N\left(\ph\right)-\eta_T\left(\ph\right)=\frac{\gamma_T(\un_F)}{\gamma_T^N(\un_F)} \frac{1}{p_T} \gamma_T^N(\ph-\eta_T\left(\ph\right))$$
and the fact that $\gamma_T^N(\un_F)$ goes in probability to $\gamma_T(\un_F)$ (see Proposition \ref{pro:estimate}).

\begin{Cor}\label{eta}
Under Assumption~(SK), for any $\ph$ in $\overline{\cal D}$, one has
$$\sqrt{N}\left(\eta_T^N(\ph)-\eta_T(\ph)\right)\xrightarrow[N\to\infty]{\cal D}{\cal N}(0,\sigma_T^2(\ph-\eta_T(\ph))/p_T^2).$$
\end{Cor}

In the next subsection, we propose to focus our attention on the estimator $p_T^N$ in order to discuss the asymptotic variance given by Theorem \ref{gamma}.

\subsection{Some comments on the asymptotic variance}

In this section, we assume that the function $\un_F$ satisfies Assumption (CC). 
Then, taking $\ph=\un_F$ in Theorem \ref{gamma} yields
$$\sqrt{N}\left(p_T^N-p_T\right)\xrightarrow[N\to\infty]{\cal D}{\cal N}(0,\sigma_T^2(\un_F)),$$
where
\begin{equation}\label{lakzch}
\sigma_T^2(\un_F) = - p_T^2\ln(p_T) - 2\int_0^T \V_{\eta_{t}}(Q^{T-t}(\un_F)) p_t p'_t  \d t.
\end{equation} 
In this expression, notice that
\begin{equation}\label{wxoicj}
\V_{\eta_{t}}(Q^{T-t}(\un_F))=\V(\P(X_T\neq\partial |X_t))=\E\left[\left(\P(X_T\neq\partial |X_t)-\frac{p_T}{p_t}\right)^2\right].
\end{equation}
Here $\P(X_T\neq\partial |X_t)$ is a random variable with values between 0 and 1, and expectation $p_T/p_t$. Hence the maximal possible value for the variance is obtained for a Bernoulli random variable with parameter $p_T/p_t$, so that
$$0\leq\V_{\eta_{t}}(Q^{T-t}(\un_F))\leq \frac{p_T}{p_t} \left(1-\frac{p_T}{p_t}\right).$$
Taking into account that $p'_t\leq 0$, we finally get the following bounds for the asymptotic variance of the probability estimate: 
\begin{equation}\label{lezch}
- p_T^2\ln(p_T)\leq\sigma_T^2(\un_F)\leq 2p_T(1-p_T)+p_T^2\ln(p_T).
\end{equation}

According to (\ref{wxoicj}), the lower bound is reached when, for each $t\in[0,T]$, the probability of being still alive at time $T$ is constant on the support of the law $\eta_t$. This situation includes, but is not limited to, the trivial case where the killing intensity $\lambda(x)$ is constant and equal to $\lambda$ on the whole space $F$. Then, for any initial condition, $\tau_\partial$ has an exponential distribution with parameter $\lambda$ and, obviously, $\V_{\eta_{t}}(Q^{T-t}(\un_F))=0$. In fact, in this elementary framework, one can be much more precise about the estimator 
$$p_T^N=\left(1-\tfrac{1}{N}\right)^{B_T}.$$
Indeed, a moment thought reveals that $(B_t)_{t\geq 0}$ is just a Poisson process with intensity $N\lambda$, so that $B_T$ has a Poisson distribution with parameter $N\lambda T$, and $p_T^N$ is a discrete random variable with law
 $$\P\left(p_T^N=\left(1-\tfrac{1}{N}\right)^k\right)=e^{-N\lambda T}\frac{(N\lambda T)^k}{k!}\hspace{1cm}\forall k\in\N.$$
 In particular, it is readily seen that this estimator is unbiased:
 $$\E[p_T^N]=e^{-\lambda T}=\P(X_T\neq\partial)=p_T,$$
 with variance 
 $$\V(p_T^N)=(p_T)^2\left(e^{\lambda T/N}-1\right)\ \Longrightarrow\ \lim_{N\to\infty}N \V(p_T^N)=-p_T^2\ln(p_T),$$
 which is exactly the lower bound in (\ref{lezch}).\medskip
 
By contrast, the upper bound in (\ref{lezch}) may be surprising at first sight. Indeed, notice that the crude Monte Carlo estimator $\hat{p}_T^N$ described in Section \ref{intro} satisfies
$$\sqrt{N}\left(\hat{p}_T^N-p_T\right)\xrightarrow[N\to\infty]{\cal D}{\cal N}(0,p_T(1-p_T)).$$ 
As $2p_T(1-p_T)+p_T^2\ln(p_T)\geq p_T(1-p_T)$ for any $p_T\in[0,1]$, this suggests that there are some situations where the Fleming-Viot estimator is less precise than the crude Monte Carlo estimator. More precisely, if $p_T$ is small, then $2p_T(1-p_T)+p_T^2\ln(p_T)\approx 2p_T(1-p_T)$, that is almost twice less precise in terms of asymptotic variance.\medskip

Although counterintuitive, this phenomenon can in fact be observed on a toy example. Take $F=\{0,1\}$ for the state space, $\eta_0=p\delta_0+(1-p)\delta_1$ with $0<p<1$ for the initial distribution, $\lambda_0=0<\lambda_1$ for the absorbing rates, and consider the process $X_t=X_0$ until time $\tau_\partial$. In other words, nothing happens before killing and the process can be killed if and only if $X_0=1$. Suppose that our goal is to estimate the probability $p_1$ that the process is still alive at time $T=1$. Clearly, for all $t\geq 0$, one has $p_t=p+(1-p)\exp(-\lambda_1 t)$ and the law of $X_t$ given that the process is still alive at time $t$ writes
$$\eta_t=\frac{1}{p_t}\left(p\delta_0+(1-p)\exp(-\lambda_1 t)\delta_1\right).$$
Since, for any $t\in[0,1]$, 
$$\P(X_1\neq\partial |X_t)=\un_{X_t=0}+\exp(-\lambda_1(1-t))\un_{X_t=1},$$
we deduce that
$$\V_{\eta_{t}}(Q^{T-t}(\un_F))=\frac{p+(1-p)e^{-\lambda_1(2-t)}}{p_t}-\left(\frac{p_1}{p_t}\right)^2=\frac{(p_1-p)^2}{p_t(p_t-p)}+\frac{p}{p_t}-\left(\frac{p_1}{p_t}\right)^2.$$ 
Therefore, taking $T=1$ in (\ref{lakzch}), the asymptotic variance is equal to
$$\sigma_1^2(\un_F)=2p(1-p_1)+2(p_1-p)^2\ln\frac{1-p}{p_1-p}+p_1^2\ln p_1.$$ 
Finally, remark that $p_1$ can be made arbitrarily close to $p$ by taking $\lambda_1$ sufficiently large, which in turn leads to a variance that is arbitrarily close to the upper bound in (\ref{lezch}).\medskip

Therefore, the take-home message is that we can easily exhibit pathological examples where the application of Fleming-Viot particle systems is in fact counterproductive compared to a crude Monte Carlo method.  Intuitively, the branching process in Fleming-Viot simulation improves the focus on rare events, but creates a strong dependency between trajectories.

\subsection{Comparison with the discrete time case}

In this section we will compare our results with what happens with the following discrete time algorithm. We start with a given finite set of times $t_0=0<t_1<\dots<t_n=T$. Let us assume to simplify that the $t_j$'s are evenly spaced in terms of survival probability, that is $p_{t_j}/p_{t_{j-1}} = p(n)$ for all $j$, with $p(n)\rightarrow 1$ when $n\rightarrow +\infty$.\medskip

We start with $N$ independent copies of the process $X$ and run them until time $t_1$. The ones having reached $\partial$ are then killed, and for each one killed, we randomly choose one that is not and duplicate it. Then we run the new and old (not killed) trajectories until time $t_2$, and iterate until we reach time $t_{n}=T$. If at some point all the trajectories are killed, i.e.~they all have reached $\partial$, then we consider that the run of the algorithm has failed and we call this phenomenon an extinction.\medskip

This discrete version of the algorithm falls in the framework of~\cite{delmoral04a}, so we can apply the results therein.  Among various convergence results, we will specifically focus on CLT type theorems and compare them to our setting. Let us also mention that the extinction probability is small when $N$ is large: specifically, there exist positive constants $a$ and $b$ such that the probability of extinction is less than $a\exp(-N/b)$ (Theorem 7.4.1 in \cite{delmoral04a}).\medskip

At each $t_k$, we denote by $\tilde\eta^N_k$ the empirical measure of the particles just before the resampling. We can estimate the probability $\P(\tau_\partial>T)$ by 
$$\prod_{k=1}^{n} \tilde\eta^N_k(\un_F)=\tilde\gamma^N_n(\un_F)\tilde\eta^N_n(\un_F)\hspace{1cm}\mbox{with}\hspace{1cm}\tilde\gamma^N_n(\un_F)=\prod_{k=1}^{n-1} \tilde\eta^N_k(\un_F).$$ 
We also define the unnormalized measures through their action on test functions $\ph$ by $\tilde\gamma^N_n(\varphi)=\tilde\gamma^N_n(\un_F)\tilde\eta^N_{n}(\varphi)$. As previously, we will assume that $\ph(\partial)=0$, which implies that for all $t\geq 0$, $Q^t(\ph)(\partial)=0$. The following CLT is then a straightforward generalization of Theorem 12.2.2 and the following pages of \cite{delmoral04a} : 
$$
\sqrt{N}\left(\un_{\tau^N>n}\tilde\gamma_n^N(\ph)-\gamma_T(\ph)\right)\xrightarrow[N\to\infty]{\cal D}{\cal N}(0,\tilde\sigma_n^2(\ph)),
$$
with $\tau^N$ the extinction iteration of the particle system, and $\tilde\sigma_n^2(\ph) = a_n - b_n$, where
$$a_n=\eta_0((Q^T\ph - \eta_0(Q^T\ph))^2)+\sum_{j=1}^n \gamma_{t_{j-1}}(\un_F)^2 \tilde\eta_{t_j}((Q^{T-t_j}\ph-\tilde\eta_{t_j}(Q^{T-t_j}\ph))^2),$$
and
$$b_n=\sum_{j=1}^n \gamma_{t_{j-1}}(\un_F)^2 \tilde\eta_{t_{j-1}}(\un_F(Q^{T-t_{j-1}}\ph-\tilde\eta_{t_j}(Q^{T-t_j}\ph))^2),$$
with $\tilde\eta_{t_j}=p(n) \eta_{t_j}+(1-p(n))\delta_\partial$. We do not have exactly $\eta_{t_j}$ because it is an updated measure, while the CLT of \cite{delmoral04a} applies to predicted measures (see \cite{delmoral04a} Sections 2.7.1 and 2.7.2 for a discussion on the difference).
%where $Q_X$ is the semi-group assiociated with the (non killed) process $X$. 
After some very basic algebra, this asymptotic variance can be written as
\begin{align*}
\tilde\sigma_n^2(\ph) = \sum_{j=1}^n  \gamma_{t_{j-1}}(\un_F)^2 \Big{(} & p(n)(\eta_{t_j}((Q^{T-t_j}\ph)^2) - \eta_{t_{j-1}}((Q^{T-t_{j-1}}\ph)^2))\\
&-2p(n)^2\eta_{t_j}(Q^{T-t_j}\ph)(\eta_{t_j}(Q^{T-t_j}\ph)-\eta_{t_{j-1}}(Q^{T-t_{j-1}}\ph))\\
&+p(n)^2\eta_{t_j}(Q^{T-t_j}\ph)^2(1-\eta_{t_{j-1}}(Q^{t_j-t_{j-1}}\un_F)) \Big{)}\\
+\eta_0((Q^T\ph - \eta_0&(Q^T\ph))^2).
\end{align*}
Now, we should remember that $p_t=\gamma_t(\un_F)$, and that
$$1-\eta_{t_{j-1}}(Q^{t_j-t_{j-1}}\un_F) = 1-\frac{\gamma_{t_j}(\un_F)}{\gamma_{t_{j-1}}(\un_F)} = \frac{p_{t_{j-1}}-p_{t_j}}{p_{t_{j-1}}}.$$
If we make $n\rightarrow \infty$, which implies that  $\sup_{j}(t_j-t_{j-1})\rightarrow 0$, we have, at least formally, that $\tilde\sigma_n^2(\ph) \rightarrow \tilde\sigma_\infty^2(\ph)$ with
\begin{align*}
 \tilde\sigma_\infty^2(\ph) = & \int_0^T p_t^2 \frac{d}{dt} (\eta_t((Q^ {T-t}\ph)^2))  \d t -2\int_0^T p_t^2 \eta_t(Q^{T-t}\ph)\frac{d}{dt}( \eta_t(Q^{T-t}\ph))  \d t\\
 &-\int_0^T p_t\eta_t(Q^{T-t}\ph)^2 p_t'  \d t+\eta_0((Q^T\ph - \eta_0(Q^T\ph))^2).
\end{align*}
By integrating by parts the first two integrals, and noticing in the third one that $$p_tp_t' \eta_t(Q^{T-t}\ph)^2 = p_T^2\eta_T(\ph)^2 \frac{p_t'}{p_t},$$ we get that 
$$
\tilde\sigma_\infty^2(\ph) = p^2_T \V_{\eta_T}(\ph) - p_T^2\ln(p_T) \, \eta_T(\ph)^2 - 2\int_0^T \V_{\eta_{t}}(Q^{T-t}(\ph)) p_t p'_t  \d t,
$$
which is exactly the simplified expression of  $ \sigma_T^2(\ph)$,  the asymptotic variance of Theorem~\ref{gamma}.\medskip

In other words, the asymptotic variance of the continuous time algorithm can be interpreted as the limit of the asymptotic variance of the discrete time algorithm, when the time mesh becomes infinitely fine, i.e.~when the number of resamplings goes to infinity.

\subsection{Some comments on the assumptions}\label{pmazioch}

\subsubsection{Discussion}\label{maoezcj}

Assumptions (SK) and (CC) are related to the so-called ``generator'' of the Markov process $t \mapsto X_t$. \medskip

A bounded time-dependent function $(x,t) \mapsto \psi(t,x)$ is said to belong to the domain of the \emph{extended generator} of $t \mapsto X_t$ if there exists a \emph{bounded} function formally denoted $\p{\partial_t+L}(\psi)$ - or simply $L(\psi)$ if $\psi$ is independent of time - such that, for any initial $\calL(X_0)$, the process
\[
t \mapsto \psi(t,X_t) - \psi(0,X_0) - \int_0^t \p{\partial_s +L(\psi)}(s,X_s) \d s 
\]
is a martingale with respect to the considered right-continuous filtration for which $X$ satisfies the Markov property. Note that $t \mapsto \psi(t,X_t)$ has then a \cadlag modification. The following sufficient criterion can be checked easily:
\begin{Lem}\label{lem:gen}
Let $\psi$ and $\p{\partial_t+L}(\psi)$ be bounded time-dependent functions defined for $t \in [0,T]$. If for any initial $\calL(X_0)$, and for each $ 0 \leq t \leq T$, one has
\begin{equation}\label{eq:def_gen}
 \E \b{ \p{\partial_t+L}(\psi)(t,X_t)} = \frac{\d}{\d t} \E \b{\psi(t,X_t)},
 %\frac{\eta Q^h(\psi(t+h,\, .))- \eta(\psi(t, \, .))}{h}
 %= \lim_{h\downarrow 0}\frac{\E(\psi(t+h,X_{t+h}))-\E(\ph(t,X_t))}{h},
\end{equation}
then $\psi$ belongs to the domain of the extended generator of $t \mapsto X_t$.
\end{Lem}

By definition, Assumptions~(SK) and~(CC) can thus be checked thanks to Lemma~\ref{lem:gen}. Indeed, Assumption~(SK) is equivalent to the fact that $\un_F$ belongs to the extended generator of $t \mapsto  X_t$, in which case $\lambda = -L(\un_F)$. In the same way, $\ph$ satisfies Assumption~(CC) if and only if the time-dependent function $(t,x) \mapsto \b{ Q^{T-t}(\ph) }^2(x)$ belongs to the extended generator of $t \mapsto (t,X_t)$ for $t \in [0,T]$. % and each $\eta_0 = \calL(X_0)$.% = \lim_{h\downarrow 0} \frac{\eta_0 Q^h(\ph)(x)- \eta_0(\ph)}{h}$
%belongs to the domain of the extended generator - denoted formally  $\partial_t+L$ - of the Markov process $t \mapsto (t,X_t)$ as defined by~\eqref{eq:def_gen}. 
Then one has
\begin{align*}
 \Gamma_{T-t}(\ph) = \p{\partial_t + L} \b{\p{Q^{T-t}(\ph)}^2}.
\end{align*}

For the specific case of Feller semi-groups recalled in Section~\ref{sec:fellerannex}, the domain of the infinitesimal generator is contained in the domain of the extended generator. This may be useful in practice in order to check Assumptions~(SK) and (CC).\medskip

Assumption~(CC) is also related to the so-called ``carré du champ'' of the Markov process $t \mapsto X_t$ through the \emph{formal} formula:
\[
 \Gamma_{t}(\ph)(x) = \Gamma\p{Q^{t}(\ph),Q^{t}(\ph)}(x),
\]
where the carré du champ operator $\Gamma$ is defined by $\Gamma(\ph,\ph)=L(\ph^2)-2\ph L(\ph)$, or alternatively by the forward in time variance formula $$\Gamma(\ph,\ph)(x)\eqdef  \lim_{h\downarrow 0}\frac{\V(\ph(X_h)|X_0=x)}{h}.$$ 
Further comments on the carré du champ operator are given in Section~\ref{sec:carre}.

\subsubsection{Some examples}\label{sec:examples}

This section illustrates assumptions (SK) and (CC) in two elementary but typical situations, namely Piecewise Deterministic Markov Processes and diffusions. The interested reader will find more comments and details on these assumptions in the appendix, especially concerning the case of Feller processes (see Section \ref{sec:fellerannex}).

\paragraph{Piecewise Deterministic Markov Processes} 

Our goal here is to show that Assumptions~(SK) and~(CC) are satisfied by Piecewise Deterministic Markov Processes with bounded jump intensity. Let us first explicit the framework we have in mind for this example. As we will see, it encompasses the ruin problem mentioned in Section \ref{intro}.

\begin{Def}\label{def:PDMP}
 A Piecewise Deterministic Markov Process (PDMP) with killing on state space $F \cup \set{ \partial } $ with $\partial \notin F$, with bounded jump intensity and bounded killing intensity is uniquely defined in law by:
 \begin{itemize}
  \item a measurable deterministic flow $(t,x) \in \R^+ \times F  \mapsto \psi_t(x) \in F$ satisfying the semi-group property $\psi_s \circ \psi_t = \psi_{s+t}$ and such that $t \mapsto \psi_t(x)$ is c\`adl\`ag;
  \item a non-negative kernel $q(x,dx')$ from $F$ to $F \cup \set{ \partial } $ satisfying the boundedness assumption $ \sup_x \int_{F \cup \set{ \partial } } q(x,dx') < + \infty $ and, by convention, $ q(x,\{x\}) = 0$.
%  \item A killing intensity $\lambda$ satisfying $ \sup_x \lambda(x) < +\infty $.
 \end{itemize}
This PDMP can be sequentially constructed as follows. Define
$$\bar{q}(x)\eqdef\int_{F \cup \set{ \partial }}  q(x,dx'),$$ 
let $(E_n)_{n \geq 0}$ denote a sequence of i.i.d.~unit mean exponential random variables, and $X_0 \in F$  a given independent initial condition.
\begin{enumerate}[(i)]
  \item The random jump times $(T_n)_{n \geq 0}$ are defined by $T_0 = 0$ and, for all $n\geq 0$,
$$\dps \int_{T_n}^{T_{n+1}} \bar{q}(X_t) \d t = E_{n+1};$$  
  \item For $t \in [T_n,T_{n+1})$, set $X_t \eqdef   \psi_{t-T_n}(X_{T_n})$;
  \item For any $n\geq 0$, set $X_{T_{n+1}}\sim q(X_{T_{n+1}^-},dx') / \bar{q}(X_{T_{n+1}^-})$;
  \item The killing time $\tau_\partial$ is defined by $\tau_\partial\eqdef\inf \set{T_n,\ X_{T_n} = \partial}$. %\dps \int_0^{\tau_\partial} \lambda(X_t) \d t = E_0 $.
\end{enumerate}
\end{Def}

Such processes are c\`adl\`ag Markov processes (and even strong Markov, see e.g.~\cite{davis84}). It turns out that they satisfy the previous assumptions, as stated by the next result, whose proof is postponed to Section~\ref{amoch}.

\begin{Pro}\label{pdmp}
The Markov process specified by Definition~\ref{def:PDMP} satisfies:
\begin{itemize}
\item Assumption~(SK) with killing intensity $\lambda(x) = q(x, \set{\partial})$,
\item Assumption~(CC) for any bounded $\ph$, with carr\'e du champ defined by $\Gamma_t(\ph) = \Gamma\p{Q^t(\ph),Q^t(\ph)}$ and
$$\Gamma(\ph,\ph)(x)\eqdef\int_{x' \in F} \p{\ph(x') - \ph(x)}^2 q(x,dx').$$
\end{itemize}
\end{Pro}

The ruin problem mentioned in Section \ref{intro} can be seen as a PDMP satisfying these assumptions. Namely, suppose for simplicity that the claim size $Y_1$ with cdf $F_Y$ has a density $f_Y$ on $F=\R_+$ with respect to Lebesgue's measure, then for all $x$ and $x'$ in $F=\R_+$, we have
$$\left\{\begin{array}{rl}
q(x,dx')=&\theta  f_Y(x-x') \un_{0\leq x'\leq x}dx'\\
q(x,\{\partial\})=&\theta (1-F_Y(x))\\
\bar{q}(x)=&\theta\\
\psi_t(x)=&x+ct\\
\lambda(x)=&\theta (1-F_Y(x))
\end{array}\right.$$
Clearly, for any $x\geq 0$, the mapping $t \mapsto \psi_t(x)$ is continuous and satisfies $\psi_s \circ \psi_t = \psi_{s+t}$, the non-negative kernel $q$ is such that $q(x,\{x\})=0$ and 
$$\sup_x \int_{F\cup \set{ \partial }} q(x, \d x ')=\theta < + \infty,$$
and the killing intensity is bounded as well since  $\|\lambda\|_\infty\leq\theta<+\infty$. Hence, in this case, the carr\'e du champ is defined by
$$\Gamma(\ph,\ph)(x)\eqdef\theta\int_F \p{\ph(x') - \ph(x)}^2 f_Y(x-x') dx'.$$

\paragraph{Diffusions}

Let $(Z_t)_{t \geq 0}$ denote a diffusion on $\R^d$ defined for all time $ t\geq 0$ and with - formal - generator
\begin{align*} 
L_0=\frac12\sum_{i,j}a_{ij}(x)\frac{\partial^2}{\partial x_i\partial x_j}
+\sum_{i}b_{i}(x)\frac{\partial}{\partial x_i}.
\end{align*}
We assume that for each initial distribution $\calL(Z_0)$, the latter is a weak solution to the Stochastic Differential Equation (SDE)
\begin{align}\label{eq:SDE}
\d Z_t=b(Z_t)\d t+\sigma(Z_t) \d B_t
\end{align}
where $\sigma \sigma^T =  a$. In~\eqref{eq:SDE}, $b$ and $\sigma$ are at least locally bounded, $(Z_t)_{t \geq 0}$ is adapted and $(B_t)_{t \geq 0}$ is a standard Brownian motion with respect to some unprescribed filtration, see~e.g.~\cite{ek86,RogWil_vol2}. \medskip 

The process $X_t$ is defined by $X_t=Z_t$ for $t < \tau_\partial$, and $X_t = \partial$ for $t \geq \tau_\partial$. The killing happens at a locally bounded rate $\lambda$, so that $\int_0^{\tau_\partial} \lambda(Z_t) \d t = - \log U$ where $U \in [0,1]$ is uniform and independent of $(Z_t)_{t \geq 0}$. One then has the Feynman-Kac formula
\begin{align} \label{qtphi}
Q^t(\ph)(x) = \E_x \b{\ph(X_t)}=\E_x \b{\ph(Z_t) e^{- \int_0^t \lambda(Z_s) \d s } },
\end{align}
where by convention $\ph( \partial)=0$. The generator of $(Q^t)_{t \geq 0}$ is then formally given by $L(\ph) = L_0(\ph) -\lambda \ph $. A quick formal calculation also enables to define the carr\'e du champ as
\[
 \Gamma(\ph,\ph) =L(\ph^2) - 2 \ph L(\ph) =  \sum_{i,j} a_{ij}(x)\partial_{x_i} \ph \partial_{x_j} \ph + \lambda \ph^2.
\]
In the special case of a Brownian motion with $\lambda = 0$, this implies that $\Gamma(\ph,\ph)=|\nabla \ph|^2$, hence the name ``carr\'e du champ'' (square field).\medskip

In this setting, the next result is straightforward.
\begin{Lem}
If $\lambda$ is bounded, then Assumption~(SK) is satisfied. 
\end{Lem}

 %We want to prove that the function $\un_{\R^d}$ belongs to the domain of $L$ considered as an extended generator, with $L\un_{\R^d}=-\lambda$. 
 Indeed, from the Feynman-Kac formula and Lemma~\ref{lem:gen}, it is sufficient to verify that
\[
 \frac{\d}{\d t} \E \b{ e^{- \int_0^t \lambda(Z_s) \d s } } = \E \b{ -\lambda(Z_t)  e^{- \int_0^t \lambda(Z_s) \d s } },
\]
which is just a dominated convergence result when $\lambda$ is bounded.\medskip

Assumption~(CC) can typically be checked using It\^o calculus, provided that $(x,t) \mapsto Q^t(\ph)(x)$ is regular enough. The key formula here is the following representation of the \cadlag version of the martingale
\begin{equation}\label{eq:key_diff}
\M_t(\ph) \eqdef Q^{T-t}(\ph)(X_t) = \int_0^t \b{\nabla_x Q^{T-s}(\ph) \sigma }(X_s) \d B_s + J_t
\end{equation}
where $J$ is a martingale defined by
\[
J_t \eqdef \M_{t^-}(\ph) \un_{X_t \in \partial} - \int_0^{t} \M_s(\ph) \lambda(X_s) \d s .
\]
For instance, one can obtain easily:
\begin{Lem}
Assume that $\sigma$ and $\lambda$ are bounded, and that there exists a function $(t,x) \mapsto \nabla_x Q^t(\ph)(x)$, bounded on $[0,T] \times \R^d$, such that~\eqref{eq:key_diff} holds true. Then $\ph$ satisfies Assumption~(CC) with
\[
 \Gamma_t(\ph) = \Gamma( Q^t(\ph), Q^t(\ph)  ).
\]
\end{Lem}
\begin{proof} Denoting 
$$L_t \eqdef  \int_0^t \b{\nabla_x Q^{T-s}(\ph) \sigma }(X_s) \d B_s,$$
remark that (i) by construction $L$ and $J$ are orthogonal, (ii) by It\^o calculus $$ \d \bracket{L,L}_t =  \sum_{i,j}\b{a_{i,j}\partial_{x_i} Q^{T-t}(\ph)\partial_{x_j} Q^{T-t}(\ph) }( X_t ) \d t, $$ 
(iii) by construction of the jump, $ \d \bracket{J,J}_t = \b{\lambda \p{Q^{T-t}(\ph)}^2} (X_t) \d t$.
\end{proof}

According to Remark~$1$, Theorem~$3.3$, Chapter~IV in~\cite{ry99}, if the function $(t,x) \mapsto Q^{T-t}(\ph)(x)$ belongs to $C^{1,2}\p{[0,T]\times \R^d}$, one can apply It\^o's lemma to the process $t \mapsto
Q^{T-t}(\ph)(X_t)$ to obtain $\p{-\partial_t+L }\p{Q^{T-t}(\ph)}=0$ and then key formula~\eqref{eq:key_diff}. \medskip

%Much less regularity on $(t,x) \mapsto Q^{T-t}(\ph)(x)$ is required to obtain~\eqref{eq:key_diff} by using an approximation argument. For instance, assume that there exists a sequence of smooth functions $\psi_n \in C^\infty_c([0,T] \times \R^d)$ such that $\psi_n$ and $\nabla_x  \psi_n$ are, (i) uniformly - with respect to $n$ - bounded  and,  (ii) converge pointwise to $Q^{T-t}(\ph)$ and $\nabla Q^{T-t}(\ph)$ respectively. Then it is possible to check by applying It\^o formula to the process $t \mapsto \psi_n(t,X_t)$, that~\eqref{eq:key_diff} holds true using a dominated convergence argument. \medskip

In order to check the regularity of $(t,x) \mapsto Q^{t}(\ph)(x)$, and especially the boundedness of $(t,x) \mapsto \nabla_x Q^{t}(\ph)(x)$, one has to distinguish between: on the one hand, results obtained from Partial Differential Equations techniques, typically in the elliptic case; and, on the other hand, results obtained from stochastic flows generated by strong solutions of the considered SDE.\medskip

%For instance, if (i) $a$ satisfies a uniform ellipticity condition, and (ii) $a$, $b$, $\lambda$ and $\ph$ are bounded and satisfy a global $\alpha$-H\"older condition, some standard -yet technically difficult and lengthy - analysis results in H\"older spaces, usually referred to as Schauder estimates, enables to check that $(t,x) \mapsto Q^{t}(\ph)(x)$ is $\alpha$-H\"older continuous and bounded up to order $1$ in time and $2$ in space. See some details in: \cite{Fried}, \cite{dynkin2} Th.~$0.4$, **** COMPLETER AMELIORER ********. In the case where all the coefficients have higher regularity, analysis in Sobolev space is also classical; see e.g.~ Section~$7.3.1$ of~\cite{EvansPDE}. \medskip

For instance, if (i) $a$ satisfies a uniform ellipticity condition, and (ii) $a$, $b$, 
$\lambda$ and $\ph$ are bounded and satisfy a global $\alpha$-H\"older condition, 
classical estimates of the transition probability function
enable to check that $(t,x) \mapsto Q^{t}(\ph)(x)$ is $\alpha$-H\"older continuous and bounded up to order $1$ in time and $2$ in space (see for example \cite{dynkin2} Theorem $0.4$). \medskip

Concerning the method based on stochastic flows, which does not require any ellipticity assumption, the idea is to consider strong solutions of the SDE 
$$\left\{
\begin{array}{lrl}
Z^x_0&=&x \\
dZ^x_t&=&b(Z^x_t)dt+\sigma(Z^x_t)dB_t,
\end{array}
\right.$$
where $b$ and $\sigma$ are at least globally Lipschitz. Under additional regularity of the coefficients, the derivative with respect to the initial condition is solution of another SDE, and some a priori estimates may lead to a probabilistic representation formula for $\nabla_x Q^{T-t}(\ph)$ and give an upper bound on the latter. We refer to \cite{Gikhman} Section 8.4, \cite{krylov} Section 2.8, or to \cite{skorokhod1981}.

%----------------------------------------------------------------------
%----------------------------------------------------------------------
\section{Proof}\label{mazlco}
%----------------------------------------------------------------------
%----------------------------------------------------------------------

As in \cite{v14}, the key objects of the proof are the bounded c\`adl\`ag martingales
\begin{equation}\label{eq:M_t}
t \mapsto M_t\eqdef   M_t(\ph)= \gamma^N_t \p{Q^{T-t}(\ph) },  
\end{equation}
verifying
$$M_T -M_0= \gamma^N_T( \ph) - \gamma^N_0(Q^T(\ph))$$ 
which, formally, is converging to $0$ when $N \to \infty$ by the law of large numbers. This point was already established in \cite{v14} and will be recalled in Proposition \ref{pro:estimate}. However, our ultimate goal here is to prove a CLT on $\gamma^N_T( \ph)-\gamma_T( \ph)$. This requires several intermediate steps that we propose to detail in the upcoming subsections.

\subsection{Background on stochastic calculus}\label{remind}

Let us recall some technical tools that will prove useful in the sequel. We refer the reader to \cite{dm82,js03,ry99} for details and complements on stochastic calculus.

\paragraph{Filtrations and stopping times}

Recall that $\tau \in [0,+\infty]$ is a stopping time if $\set{\tau \leq t} \in \calF_t$ for any $t \geq 0$. Denoting $\calF_{\infty} \eqdef \mathop{\bigvee}_{t \geq 0} \calF_t$, $\calF_\tau$ is defined by the following property: $A \in \calF_\tau \subset \calF_{\infty}$ if and only if $A \cap \set{\tau \leq t} \in \calF_t$ for any $t \geq 0$.\medskip 

Next, $\calF_{\tau-} \subset \calF_{\tau}$ is defined as the smallest $\sigma$-field containing $\calF_0$ and all $B \cap \set{t < \tau}$ where $B \in \calF_t$ and $t\geq 0$. Note that $\tau$ is then $\calF_{\tau-}$-measurable. These properties still hold true for stopping times: if $\sigma$ is another stopping time, then $A \in \calF_\tau$ implies $A \cap \set{\tau \leq \sigma} \in \calF_\sigma$, and $ B \in \calF_\sigma$ implies $  B \cap \set{\sigma < \tau} \in \calF_{\tau^-}$.\medskip  

Finally, recall that if $\tau_n,\, n \geq 1$, is a sequence of stopping times, then: (i) $\sigma : = \inf_n \tau_n$ is a stopping time, and if the filtration is right-continuous, $\calF_{\sigma} = \mathop{\bigcap}_{n \geq 0} \calF_{\tau_n}$; (ii) $\tau : = \sup_n \tau_n$ is a stopping time, and if for all $n$, $\tau_n < \tau$ on the event $\set{\tau > 0}$, then $\calF_{\tau^-} = \mathop{\bigvee}_{n \geq 0} \calF_{\tau_n}$.

\paragraph{Predictability}
A stopping time $\sigma \in [0,+\infty] $ is said to be predictable, if there exists an increasing sequence of stopping times $\sigma_n \uparrow \sigma$ such that for all $n \geq 0$,
$\sigma_n < \sigma$ on the event $\set{\sigma > 0}$. Such a sequence is called an announcing sequence of stopping times. A process $(X_t)_{t \geq 0}$ with right and left limits is predictable if and only if there exists a sequence of predictable stopping times $ \sigma_n, \, n \geq 0$, such that: (i) the sequence exhausts the times where $X$ is not left-continuous, i.e. $X_{t-} \neq X_t \Rightarrow t \in \set{\sigma_n, \, n \geq 1}$; and (ii) for each $n$, $X_{\sigma_n}$ is $\calF_{\sigma_n^-}$-measurable. In particular, all left-continuous processes, and all right-continuous counting processes of predictable stopping times are predictable. Note that if $X$ is predictable, then $X_\tau$ is $\calF_{\tau^-}$-measurable. Nonetheless, this property is not sufficient to ensure predictability.

\paragraph{Martingales and stochastic integrals}
Let us recall that a property of a random process holds ``locally'' if it holds for the process stopped at $\tau_n$ for each $n \geq 1$, where $ \tau_n$ is some sequence of stopping times increasing to infinity. Stochastic integrals of the form
\[
 \int P_{t} \d X_t
\]
make sense when $t \mapsto P_t$ is predictable and locally bounded, and $t \mapsto X_t$ is either a c\`adl\`ag local martingale, or a c\`adl\`ag monotone process. \medskip

In the present paper, the predictable integrand will most often be of the form $P_t=Y_{t^-} = \lim_{h \to 0^+}Y_{t-h}$, the left continuous version of any bounded c\`adl\`ag process $Y$. Moreover, $X$ will most often be a bounded martingale or a bounded increasing process.\medskip 

% is thus a predictable bounded integrandif not specifically mentioned, processes will be at least locally bounded c\`adl\`ag semimartingales and stochastic integrals will be of the form
% \[
%  \int Y_{t-} \d X_t
% \]
% where, as usual, $Y_{t-} = \lim_{h \to 0^+}Y_{t-h}$ is the left continuous version of $Y$ and is a predictable integrand.

\paragraph{Quadratic variations}
In what follows, we will call a \emph{time mesh} a sequence $t_0=0 \leq t_1 \leq \ldots t_n \leq \ldots$ of  times such that $\lim_n t_n =\infty$. The mesh size is defined by $\delta \eqdef \sup_n (t_{n+1}- t_n)$.\medskip% and is assumed to vanish in probability.

If $t \mapsto M_t$ is a c\`adl\`ag locally square integrable martingale, the quadratic variation $\b{M,M}_t$ is the unique c\`adl\`ag process such that
\[
 \d \b{M,M}_t = \d M_t^2 - 2 M_{t^-} \d M_t
 \]
 in the sense of stochastic integration. It can be shown that
 \[
  \b{M,M}_t = \lim_{\delta \to 0}^\P \sum_{t_n \leq t} (M_{t_{n+1}} - M_{t_n})^2
 \]
where the limit holds in probability for any sequence of time meshes whose mesh size $\delta$ goes to zero.\medskip
 
On the other hand, if $t \mapsto M_t$ is a c\`adl\`ag locally square integrable martingale, the predictable quadratic variation $\bracket{M,M}_t$ is the unique predictable c\`adl\`ag process such that
\[
t \mapsto M_t^2- \bracket{M,M}_t
\]
is again a local martingale. Equivalently, $\bracket{M,M}_t$ is the compensator of the increasing process 
$t \mapsto \b{M,M}_t$.  Its existence and uniqueness are ensured by Doob-Meyer decomposition theorem. $\bracket{M,M}_t$ can be interpreted as the maximally integrated quadratic variation of the martingale. Indeed, if the martingale is square integrable, then one can prove that
\[
  \bracket{M,M}_t = \mathop{\lim}_{\delta \to 0}^{\text{weak}-L_1} \sum_{t_n \leq t} \E\b{ (M_{t_{n+1}} - M_{t_n})^2 | \calF_{t_n^-}}
 \]
 where the limit holds weakly in $L_1(\Omega,\P)$ and for any sequence of time meshes whose mesh size $\delta$ vanishes (see e.g.~Theorem~$18$, Section~1, Chapter~VII in \cite{dm82}).\medskip
 
In the present paper, predictable quadratic variations will always be continuous processes, continuity being a sufficient condition for predictability. Note finally that if $M$ is continuous, so is $\b{M,M}$, which in turn implies predictability and thus $\b{M,M} = \bracket{M,M}$.

%Hence the generic term ``discontinuity'' might correspond to a branching or a jump.

%The strategy is in several steps: (i) we prove an $L^2(\P)$ estimate $\E \b{(\gamma^N_t(\ph)- \gamma_t(\ph))^2 } \leq c / N$ as in villemonais~\cite{v14}; (ii) we compute the predictable quadratic variation $\bracket{M,M}_t$; (iii) we use step (i) to prove convergence of $N\bracket{M,M}_t$ towards a deterministic limit; (iv) we conclude using a time continuous martingale CLT.

%----------------------------------------------------------------------
\subsection{Well-posedness}
%----------------------------------------------------------------------

By convention, if a trajectory has a discontinuity, we call it a ``jump''. This should not be confused with the more restrictive term ``branching'', which corresponds only to the case where a particle is absorbed and instantaneously branches on another one.\medskip

Let us come back to the Fleming-Viot algorithm of Definition \ref{peicj}. 
In a completely general context, as in \cite{v14}, if two or more particles are simultaneously killed, 
then we would say that the process undergoes a failure and stop the algorithm. 
As ensured by the upcoming result, almost surely this can not happen in our framework 
thanks to Assumption (SK).

\begin{Lem}\label{lem:jumps}
Under Assumption (SK), 
%Then, at branching times, non-killed particles have a continuous trajectory. Formally: for all $k \geq 1$ and $n \neq m \in \set{1, \ldots , N}$,
%  $$
%  X^m_{\tau_{n,k}^-} = X^m_{\tau_{n,k}} \qquad \as.
%  $$
%  
%  *** Mathias: je pense qu'au stade actuel, on peut enlever le point ci-dessus, qui n'est pas utilise par la suite OUI ***
%  
% In particular, 
almost surely, branchings cannot occur at the same time: $\tau_{n,k}\neq\tau_{m,j}$ almost surely for any $j,k \geq 1$ if $n\neq m$. Hence, the particle system is well-defined. % up to the stopping time $\lim_{j \to +\infty} \tau_j$. 
\end{Lem}

\begin{proof}
%*** Faire directement $\calD^m = \tau^{m,j}$ ***
Let us fix $k,j \geqslant 0$, and $m\neq n \in \set{1, \ldots , N}$. Recall that $\tau_{n,0}=\tau_{m,0}=0$ by convention. It is sufficient to prove that 
$$\P( \tau_{n,k+1} = \tau_{m,j+1} \, \text{and}  \,\tau_{m,j} \leq \tau_{n,k} < \tau_{m,j+1})=0,$$ 
since taking the countable union over $j,k \geq 0$ of such events, and using the exchangeability of particles, will yield the result.\medskip

For this, consider the stopping times $S=\tau_{m,j}\vee\tau_{n,k} < \infty$ 
and $\sigma = \tau_{m,j+1} < \infty$. 
By construction of the particle system, given $\calF_S$, 
%
%%\hspace{-5em}***CORRECTION FAITE ICI***
%
$\sigma$ is independent of the particle trajectory 
$(X_t^n)_{ \tau_{n,k}\leq t < \tau_{n,k+1}}$. 
Based on Assumption (SK), Lemma~\ref{lem:quasi-left_1} ensures that 
\begin{equation*} %\label{oauzch}
\P( \tau_{n,k+1} = \sigma |  \calF_{S} ) \un_{S < \sigma } = 0,
\end{equation*}
which straightforwardly implies that
$$\P( \tau_{n,k+1} = \tau_{m,j+1} |  \calF_{\tau_{n,k}} ) \un_{\tau_{m,j} \leq \tau_{n,k} < \tau_{m,j+1}} = 0.$$
Taking the expectation gives the desired result.

%and $S$ a stopping time such that $S \geq \tau_{n,k-1}$. 
 
 % $$
% t \mapsto \un_{t \geq\tau^{n,k}} - \int_{t \wedge \tau^{n,k-1}}^{t\wedge\tau^{n,k}} \lambda(X^n_s) \, \d s 
% $$
% is a bounded, right continuous martingale. As a consequence, for any stopping time (sous-entendu : par rapport a la filtration engendree par l'ensemble des particules) $T\geq \tau^{n,k-1}$, for any $0<h<t$, we have
% $$
%  \P( \tau^{n,k} \in (T+t-h,T+t]   |  \calF_{T} )  %\E ( B^n_{(t+h)\wedge  \tau^{n,k} } - B^n_{t \vee  \tau^{n,k-1}} |  \calF_{\tau^{n,k-1}}  ) 
%  \leqslant h \norm{\lambda}_\infty.
% $$
% Letting $h \to 0$, we infer that for any $t>0$, almost surely,
%\

% Next,  For this, let us consider the decomposition
% $$\P( \tau_{n,k}\in\calD^m)=\sum_{j=1}^\infty\P( \tau_{n,k}\in\calD^m,\tau_{m,j-1}<\tau_{n,k}\leq\tau_{m,j}).$$
% and we may rewrite the previous equation as 
% \begin{align*}
% &\P( \tau_{n,k}\in\calD^m)=\\
% &\qquad\sum_{j=1}^\infty\E\left[\P( \tau_{n,k}\in\calD^m\cap(\tau_{m,j-1},\tau_{m,j}]|\calF_{S}\vee\sigma(X_s^m,S\leq s < \tau_{m,j}))\right].
% \end{align*}
% For each $j\geq 1$, 
% 
% \eqref{oauzch}~then implies
% \begin{align}
% &\P( \tau_{n,k+1} = \tau_{m,j+1} | \calF_{S}\vee\sigma(X_s^m,S \leq s < \tau_{m,j+1})) = 0 \nonumber
% %&=\un_{S+h=\tau_{n,j+1}} \P( \tau_{n,k}=S+h|\calF_{S}\vee\sigma(X_s^m,S\leq s < \tau_{m,j+1})).\nonumber
% \end{align}
% Thus we deduce from (\ref{oauzch}) that the conditional probability is almost surely equal to zero.% Since this is true for any $j$, we have the desired result.
\end{proof}

The purpose of our next lemma is to control the number of branchings.

\begin{Lem}\label{lem:finite}
Under Assumption~(SK), for any $t\geq 0$, $\E[B_t^2] < + \infty$. In particular, $B_t$ is almost surely finite.
\end{Lem}
\begin{proof} For any $n\in \set{1, \ldots , N}$, the process $(B_t^n)_{t\geq0}$ is a counting process with intensity $(\lambda(X_t^n))_{t\geq0}$. From Assumption~(SK) we deduce that $(B_t^n)_{t\geq0}$ is stochastically upper-bounded by a Poisson process with intensity $\|\lambda\|_\infty$, so that
$$\E[(B_t^n)^2]\leq (1+\|\lambda\|_\infty t)\|\lambda\|_\infty t,$$
and
$$\E[B_t^2]=\E\left[\left(\sum_{n=1}^NB_t^n\right)^2\right]\leq N^2(1+\|\lambda\|_\infty t)\|\lambda\|_\infty t.$$
\end{proof}

%----------------------------------------------------------------------
\subsection{Martingale analysis}
%----------------------------------------------------------------------

Following~\cite{v14}, we will decompose the global martingale $M_t$ as given in~\eqref{eq:M_t} into a sum of martingales by considering each particle individually. Moreover, the contribution of the Markovian evolution of particle $n$ between branchings $k$ an $k+1$ will be denoted $t \mapsto \M^{n,k}_t$, whereas the contribution of the $k$-th branching of particle $n$ will be denoted $t \mapsto \calM^{n,k}_t$. This will lead to the martingale decomposition~\eqref{eq:decomp} in Lemma \ref{lem:mart} below. \medskip

First, let us denote $t \mapsto \M_t^n$ the sum
%is the contribution of the evolution (i.e., trajectory and death, but without branchings) of the particle with label $n$. It can be decomposed as
\[
\M_t^n\eqdef    \sum_{k=0}^{B^n_t} \M_t^{n,k},
\]
with
\[
\M_t^{n,k}\eqdef  
\begin{cases}
\dps 0 & \text{if } t < \tau_{n,k} , \\
\dps Q^{T-t}(\ph)(X_t^n) - Q^{T-\tau_{n,k}}(\ph)(X_{\tau_{n,k}}^n)  &  \text{if }  \tau_{n,k} \leq t < \tau_{n,k+1}  , \\
\dps -Q^{T-\tau_{n,k}}(\ph)(X_{\tau_{n,k}}^n)  &  \text{if }  \tau_{n,k+1} \leq t,
\end{cases}
\]
or, equivalently,
\begin{equation}\label{leicj}
\M_t^{n,k} = \un_{t<\tau_{n,k+1} }Q^{T-t}(\ph)(X_t^n)-Q^{T-t\wedge\tau_{n,k}}(\ph)(X_{t\wedge\tau_{n,k}}^n),
\end{equation}
so that 
\begin{equation}\label{eq:Mbb}
 \M_t^n =   Q^{T-t}(\ph)(X_t^n) - \sum_{k=0}^{B^n_t} Q^{T-\tau_{n,k}}(\ph)(X_{\tau_{n,k}}^n).
\end{equation}
By averaging over $n$ we also define, exactly as in~\cite{v14},
$$\M_t=\frac{1}{N}\sum_{n=1}^{N}\M_t^n.$$
Second, $t \mapsto \calM_t^n$ is the contribution due to the branching rule of the particle with label $n$:
\[
\calM_t^n\eqdef    \sum_{k=1}^{B^n_t} \calM_t^{n,k},
\]
with 
\[
\calM_t^{n,k}\eqdef     
\dps \begin{cases}0 &  \text{if } t < \tau_{n,k}, \\
\dps \frac{N-1}{N}  Q^{T-\tau_{n,k}}(\ph)(X_{\tau_{n,k}}^n)- \frac{1}{N} \sum_{m \neq n} Q^{T-\tau_{n,k}}(\ph)(X_{\tau_{n,k}}^m)   & \text{if } t \geq \tau_{n,k},
\end{cases}
\]
or, equivalently,
\begin{equation}\label{aicj}
\calM_t^{n,k}=
(1-\tfrac{1}{N})\p{Q^{T-\tau_{n,k}}(\ph)(X_{\tau_{n,k}}^n)- \frac{1}{N-1} \sum_{m \neq n} Q^{T-\tau_{n,k}}(\ph)(X_{\tau_{n,k}}^m)}\un_{t \geq \tau_{n,k}}.
\end{equation}
In the same way as before, let us denote 
$$\calM_t=\frac{1}{N}\sum_{n=1}^{N}\calM_t^n.$$
Note that what Villemonais denotes $\calM_t$  in~\cite{v14} corresponds to the same quantity times $(1-1/N)$. We start with an important remark that was also made in \cite{v14}. We give the sketch of the proof for self-completeness.

\begin{Lem}
 For all $k \geq 0$ and $n \in \set{1, \ldots, N}$, $t \mapsto \calM_t^{n,k}$ and $t \mapsto \M_t^{n,k}$ are bounded c\`adl\`ag martingales. 
\end{Lem}
\begin{proof} On the one side, the martingale property for $t \mapsto \calM_t^{n,k}$ can be checked using Lemma \ref{albcios} and by remarking that, by construction of the branching rule, for any bounded function $\psi$ we have
$$
\E \left[\left.\psi(X_{\tau_{n,k}}^n )  -  \frac{1}{N-1} \sum_{m \neq n} \psi(X_{\tau_{n,k}}^m)\right| \calF_{\tau_{n,k}^-}\right]=0
$$
where $\calF_{\tau_{n,k}^-}$ can be interpreted as the sigma-field generated by the particle system including the last jump, but not the index of the particle that is going to branch. On the other side, the martingale property for  $t \mapsto \M_t^{n,k}$ is a direct consequence of Lemma \ref{lem:mart0}.

\end{proof}

We can then remark that the martingales associated with two different particles do almost surely  not jump at the same times, nor at branching times. We adopt the classic notation $\Delta Z_t=Z_t-Z_{t-}$.
 
\begin{Lem}\label{jump2}
Under Assumptions~(SK) and~(CC), for any $1 \leq m \neq n \leq N$, 
$$\P( \exists t \geq 0, \, \Delta \M^m_t \neq 0\ \&\ \Delta \M^n_t \neq 0)= 0.$$ 
and
$$ \P( \exists  j \geq 0, \, \Delta \M^n_{\tau_{m,j}} \neq 0)= 0.$$
\end{Lem}

\begin{proof} 
The proof is similar to the one of Lemma~\ref{lem:jumps}. Let us fix $k,j \geqslant 0$, and $m\neq n \in \set{1, \ldots , N}$, and consider the stopping time $S=\tau_{m,j}\vee\tau_{n,k}$. Let us also consider the random set 
$$\calD^{m,j} \eqdef \p{ \set{t \in (S,\tau_{m,j+1}), \,  \Delta \M^{m,j}_t \neq 0}\cup \set{\tau_{m,j+1} } } \cap (0,+\infty).$$ 
If we show that 
$$\P(\exists t \in \calD^{m,j}, \, \Delta \M^{n,k}_{t} \neq 0 ) = 0,$$ 
then this will prove the claimed results using exchangeability of particles and by taking the countable union over $j \geq 0$ and $k \geq 0$ of the latter events.\medskip

% For this, let us consider again the decomposition
% $$\P(\exists t \in \calD^m, \, \Delta \M^{n,k}_{t} \neq 0 )=\sum_{j=1}^\infty\P( \exists t \in\calD^m \cap (\tau_{m,j-1},\tau_{m,j}], \, \Delta \M^{n,k}_{t} \neq 0 ).$$
Since $\M^{m,j}$ is c\`adl\`ag, we know that we can construct a sequence of stopping times $(\sigma_l)_{ l \geq 1} $ that exhausts $\calD^{m,j}$: $\set{\sigma_l, l \geq 1} = \calD^{m,j}$  (see e.g.~\cite{js03} Proposition 1.32). Moreover, by construction of the particle system, the sequence $(\sigma_l)_{ l \geq 1}$ is independent of $(X_t^n)_{\tau_{n,k} \leq t < \tau_{n,k+1}}$ given $\calF_S$.\medskip

% We can then denote , and $\sigma = \tau_{m,j+1}$. From Lemma~\ref{lem:jumps}, we know that $\tau_{n,j+1} \neq \tau_{n,k}$; while on the other hand $ \M^{n,k}_{t}$ is by definition constant for $t \in \calD^{m,j}$ if $\tau_{m,j+1} < \tau_{n,k}$. \medskip

Finally, based on Assumption~(CC), Lemma~\ref{lem:quasi-left_2} ensures that, after averaging conditionally on $\calF_{S}$ and remarking that by construction $ S < \sigma_l < +\infty $, we have
\begin{equation*}%\label{oauzch2}
\P( \Delta \M^{n,k}_{\sigma_l} \neq 0 |  \calF_{S} )  = 0.
\end{equation*}
This concludes the proof.
% The stopping time $S=$ satisfies $S \geq \tau_{n,k-1}$ and we may rewrite the previous equation as
% \begin{align*}
% &\P(\exists t \in \calD^m \cap(\tau_{m,j-1},\tau_{m,j}], \, \Delta \M^{n,k}_{t} \neq 0 )= \\
% &\sum_{j=1}^\infty\E\left[\P(\exists t \in \calD^m \cap(\tau_{m,j-1},\tau_{m,j}], \, \Delta \M^{n,k}_{t} \neq 0 |\calF_{S}\vee\sigma(X_s^m, S \leq s < \tau_{m,j}))\right].
% \end{align*}
% For each $j\geq 0 $, almost surely, 
% \begin{align}
% &\P(\exists t \in \calD^{m,j}, \, \Delta \M^{n,k}_{t} \neq 0 |\calF_{S})\nonumber\\
% &\leq \sum_{l=1}^{+\infty} \P(\Delta \M^{n,k}_{\sigma_l} \neq 0 | \calF_{S}),\nonumber
% \end{align}
% where in the above, $\set{\sigma_l, l \geq 1} = \calD^{m,j}$ is a measurable enumeration of the countable set $\calD^{m,j}$. *** citer Jacod ???**
% Given $\calF_S$, the process $(X_s^m,S\leq s < \tau_{m,j+1})$ and thus $\set{\sigma_l, l \geq 1}$ is independent of the process $(X_s^n,S\leq s < \tau_{n,k+1})$ and thus of $\M^{n,k}$. Thus we deduce that each term in the sum above is almost surely equal to zero. 
\end{proof}

We can now check that $M_t$ has the appropriate decomposition.

\begin{Lem}\label{lem:mart}
Recall the notation 
$$
M_t \eqdef \gamma_t^N(Q^{T-t}(\ph)) \qquad t \in [0,T].
$$
If Assumptions~(SK) and~(CC) are satisfied, then for all $n \in \set{1, \ldots, N}$, $t \mapsto \calM^n_t$ and $t \mapsto \M^n_t$ are square integrable c\`adl\`ag martingales, and one has
 \begin{equation}\label{eq:decomp}
 \d M_t = p^N_{t-} \p{\d \M_t+\d \calM_t}.
\end{equation}
\end{Lem}
\begin{proof} 
(i) Recall that
$$\M_t^n = \sum_{k=0}^{B^n_t} \M_t^{n,k}$$
where $\M_t^{n,k}$ is bounded since $\ph$ is bounded, and $\E[(B^n_t)^2]<\infty$ by Lemma \ref{lem:finite}. Hence $t \mapsto\M_t^n$ is a square integrable c\`adl\`ag martingale, and so is $t \mapsto\M_t$. Mutatis mutandis, the same reasoning applies for $t \mapsto\calM_t$.\medskip
 
 (ii)~By construction, \eqref{eq:decomp} is valid between branching times. At branching time $\tau_{n,k}$, we have
 \begin{align*}
M_{\tau_{n,k}}-M_{\tau_{n,k}^-} & = \gamma_{\tau_{n,k}}^N(Q^{T-\tau_{n,k}}(\ph)) - \lim_{\eps \downarrow 0} \gamma_{\tau_{n,k}-\eps}^N(Q^{T-\tau_{n,k}+\eps}(\ph))  \\
 &= p_{\tau_{n,k}^-}^N\p{ (1-\tfrac{1}{N}) \eta_{\tau_{n,k}}^N(Q^{T-\tau_{n,k}}(\ph)) - \lim_{\eps \downarrow 0} \eta_{\tau_{n,k}-\eps}^N(Q^{T-\tau_{n,k}+\eps}(\ph))}. 
 \end{align*}
 From Lemma~\ref{jump2}, we know that for $m\neq n$, almost surely,
 $$\lim_{\eps \downarrow 0} Q^{T-\tau_{n,k}+\eps}(\ph)(X^m_{\tau_{n,k}-\eps})=Q^{T-\tau_{n,k}}(\ph)(X^m_{\tau_{n,k}}).$$
 By ~\eqref{leicj}, we also know that 
 $$\lim_{\eps \downarrow 0} Q^{T-\tau_{n,k}+\eps}(\ph)(X^n_{\tau_{n,k}-\eps}) = \M^{n,k-1}_{\tau_{n,k}} -   \M^{n,k-1}_{\tau_{n,k}^-}.$$
 As a consequence,
 \begin{align*}
M_{\tau_{n,k}}-M_{\tau_{n,k}^-}=&\tfrac{N-1}{N^2} p_{\tau_{n,k}^-}^N \p{ Q^{T-\tau_{n,k}}(\ph)(X_{\tau_{n,k}}^n) -\frac{1}{N-1}\sum_{m \neq n}  Q^{T-\tau_{n,k}}(\ph)(X_{\tau_{n,k}}^m)} \\
   &   -\tfrac{1}{N} p_{\tau_{n,k}^-}^N \p{\M^{n,k-1}_{\tau_{n,k}} -   \M^{n,k-1}_{\tau_{n,k}^-}} . 
 \end{align*}
 But, by (\ref{aicj}), we also have
\begin{align*}
 &\calM^{n,k}_{\tau_{n,k}} -   \calM^{n,k}_{\tau_{n,k}^-}\\
 &\quad =\left(1-\frac{1}{N}\right)  \p{ Q^{T-\tau_{n,k}}(\ph)(X_{\tau_{n,k}}^n) -\frac{1}{N-1}\sum_{m \neq n}  Q^{T-\tau_{n,k}}(\ph)(X_{\tau_{n,k}}^m)},
\end{align*}
and, for all $j\neq k$,
$$\calM^{n,j}_{\tau_{n,k}} -   \calM^{n,j}_{\tau_{n,k}^-}=0.$$

Finally, for all $j\neq k-1$, we obviously have
$$\M^{n,j}_{\tau_{n,k}} -   \M^{n,j}_{\tau_{n,k}^-}\ =0.$$ 
Putting all things together, we finally get
\begin{align*}  
& M_{\tau_{n,k}}-M_{\tau_{n,k}^-} = \frac{1}{N}p_{\tau_{n,k}^-}^N \left\{( \calM^n_{\tau_{n,k}} -   \calM^n_{\tau_{n,k}^-}) +  (\M^n_{\tau_{n,k}} -   \M^n_{\tau_{n,k}^-})\right\}.
 \end{align*}
 
 \end{proof}

The remarkable fact is that the martingale contributions are orthogonal to each other. We recall that two local square integrable martingales are orthogonal if their product is itself a local martingale.
 
 \begin{Lem} All the bounded martingales $t \mapsto \calM^{n,k}_t$,  $t \mapsto \calM^{m,j}_t$, $t \mapsto \M^{n,k}_t$, and $t \mapsto \M^{m,j}_t$ are mutually orthogonal if either $j\neq k$ or $m\neq n$. Therefore, the predictable quadratic variation of the full martingale, which uniquely exists by Doob-Meyer decomposition, is given by the sum of the predictable quadratic variation of each contribution, that is
\begin{equation}\label{eq:quad}
\d \bracket{ M, M}_t = \frac{( p^N_{t-}  )^2}{N^2}\sum_{n=1}^N \sum_{k=0}^{B^n_t} \p{   \d  \bracket{ \M^{n,k},\M^{n,k} }_t +\d \bracket{ \calM^{n,k},\calM^{n,k} }_t }.
%\var_{\frac{N}{N-1} \eta^{N}_z - \frac{1}{N-1} \delta_{X^n_z}}(\ph_z) \d K^n_z}.  
\end{equation}
\end{Lem}

\begin{proof}
The martingales  $t \mapsto \calM^{n,k}_t$ and either $t \mapsto \calM^{m,j}_t$, or $t \mapsto \M^{m,j}_t$, do not vary at the sames times if either $j \neq k$ or $n \neq m$, so they are automatically orthogonal. In the same way, the martingales $t \mapsto \M^{n,k}_t$  and $t \mapsto \M^{n,j}_t$ do not vary at the same time if $j \neq k$, and the conclusion is the same. Therefore, it remains to check the orthogonality of
\begin{itemize}
\item[(i)] $t \mapsto \calM^{n,k}_t$ and $t \mapsto \M^{n,k}_t$;
\item[(ii)] $t \mapsto \M^{n,k}_t$ and $t \mapsto \M^{m,j}_t$ for $m \neq n$.
\end{itemize}

(i) For the first point,  we have
\begin{align*}  
\M^n_{\tau_{n,k}}\calM^n_{\tau_{n,k}}- \M^n_{\tau_{n,k}^-}\calM^n_{\tau_{n,k}^-}
&=(\M^n_{\tau_{n,k}}-\M^n_{\tau_{n,k}^-})
(\calM^n_{\tau_{n,k}}-\calM^n_{\tau_{n,k}^-})\\
+&\calM^n_{\tau_{n,k}^-}
(\M^n_{\tau_{n,k}}-\M^n_{\tau_{n,k}^-})
+\M^n_{\tau_{n,k}^-}(\calM^n_{\tau_{n,k}}-\calM^n_{\tau_{n,k}^-}). 
\end{align*}
And, as noticed at the end of the proof of Lemma \ref{lem:mart}, we have
$$(\M^n_{\tau_{n,k}}-\M^n_{\tau_{n,k}^-})(\calM^n_{\tau_{n,k}}-\calM^n_{\tau_{n,k}^-})=-Q^{T-\tau_{n,k}}(\ph)(X_{\tau_{n,k}^-}^n)( \calM^n_{\tau_{n,k}} - \calM^n_{\tau_{n,k}^-}).$$
In other words, 
$$
d(\calM_t\M_t)=\calM_{t^-}d\M_t + \M_{t^-}d\calM_t - Q^{T-t}(\ph)(X_{t^-}^n)d\calM_t.
$$
Denoting as usual
$$d[\calM,\M]_t\eqdef  \d(\calM_t\M_t)-\calM_{t_-}\d\M_t-\M_{t_-}\d\calM_t,$$
we have
$$\d \b{\calM^{n}, \M^{n}}_t =  - Q^{T-t}(\ph)(X_{t-}^n) \d \calM^{n,k}_t,$$ 
which defines a martingale, and the first point is complete.\medskip

(ii)~ Concerning the second point, by construction, on each side of the interval $I\eqdef  (\tau_{n,k} \vee\tau_{m,j},  \tau_{n,k+1} \wedge \tau_{m,j+1}]$, at least one of the martingales $\M^{n,k}_t$ and $\M^{m,j}_t$ is constant, so that we can write
$$ \d  \p{\M^{n,k}_t \M^{m,j}_t }= \un_{ t \in I} \d  \p{\M^{n,k}_t \M^{m,j}_t } + \un_{ t \notin I} \p{\M^{n,k}_{t-} \d \M^{m,j}_t + \M^{m,j}_{t-} \d \M^{n,k}_t } . $$
It thus remains to show that the first term of the r.h.s. of the latter identity is a martingale. To prove this last point, let us denote by $(\widetilde{X}^{\ell}_t)_{\ell \in  \set{1, \ldots , N}, \, t \geq 0}$ the modified particle system where particles killed after $\tau_{n,k} \vee \tau_{m,j}$ stay in the cemetery point instead of branching, i.e. $\widetilde{X}^{\ell}_t = \partial$ if $t \geq \tau_{\ell,i}\geq\tau_{n,k} \vee \tau_{m,j}$ for some index $i$, and $\widetilde{X}^{\ell}_t ={X}^{\ell}_t$ otherwise. By construction, the particles $(\widetilde{X}^{\ell}_t)_{\ell \in  \set{1, \ldots , N}, \, t \geq 0}$ are mutually independent given $\calF_{\tau_{n,k} \vee \tau_{m,j}}$. Moreover, by definition, we have
$$
 \un_{ t \in I} \d  \p{\M^{n,k}_t \M^{m,j}_t } = \un_{ t \in I} \d \p{ Q^{T-t}(\ph)(\widetilde{X}^{n}_t) Q^{T-t}(\ph)(\widetilde{X}^{m}_t)},
$$
which is indeed a martingale, as can be checked easily using conditional independence and Doob's optional sampling theorem.

% Finally, the two martingales  $\M^{n,k}_t$ and $\M^{m,j}_t$ do vary together on the time interval $ [ \tau^{n,k-1} \vee \tau^{m,j-1},  \tau^{n,k} \wedge \tau^{m,j}]$ only, so that it holds:
% $$ \d  [ \M^{n,k} , \M^{m,j} ]_t = \un_{ t \leq \tau^{n,k} \wedge \tau^{m,j}} \d [
% \widetilde{\M}^n, \widetilde{\M}^m]_t $$
% where we have defined $\widetilde{\M}^l_t\eqdef   \un_{ t > \tau^{n,k-1} \vee \tau^{m,j-1}} \p{ Q^{T-t}(\ph)(\widetilde{X}^{l}_t) - Q^{T-\tau^{n,k-1} \vee \tau^{m,j-1}}(\ph)(\widetilde{X}^{l}_{\tau^{n,k-1} \vee \tau^{m,j-1}}) } $ 
%Q^{T-t}(\ph)(\widetilde{X}^{m}_t)  ] =0,$$
%by independence, and the result follows.
\end{proof}

%----------------------------------------------------------------------
\subsection{$\L_2$ estimate}
%----------------------------------------------------------------------

Applying the same reasoning as in~\cite{v14}, it is possible to prove the next result. Let us emphasize that this result is valid for any $\ph$ in $C_b(F)$, the set of continuous and bounded functions on $F$, and not only for $\ph$ in $\cal D$.

\begin{Pro}\label{pro:estimate}
Under  Assumption (SK),  for any bounded and continuous function $\ph$, we have
\[
\E \b{ \p{ \gamma^N_T(\ph) - \gamma_T(\ph) }^2 } \leq  \frac{7 \norm{\ph}^2_\infty}{N}.
\]
\end{Pro}
\begin{proof} We consider the decomposition
\begin{align*}
\gamma^N_T(\ph) - \gamma_T(\ph)  &= \int_0^T \d \b { \gamma^N_t \p{Q^{T-t}(\ph)} } +  \gamma^N_0 Q^T(\ph) - \gamma_0 Q^T ( \ph) \\
&=\int_0^T p^N_{t^-}\ d\M_t + \int_0^T p^N_{t^-}\ d\calM_t +  \gamma^N_0 Q^T(\ph) - \gamma_0 Q^T ( \ph),
\end{align*}
together with~\eqref{eq:quad}, and we compute the different contributions to the variance of the latter.

(i)~Initial condition. Since $\gamma_0=\eta_0$ and $\gamma_0^N=\eta_0^N$, we have
$$\E\b{ \p{\gamma^N_0 Q^T(\ph) - \gamma_0 Q^T ( \ph)}^2  }=\frac{1}{N}\V_{\eta_0}(Q^T(\ph)(X)),$$ 
and it is readily seen that
$$\E\b{ \p{\gamma^N_0 Q^T(\ph) - \gamma_0 Q^T ( \ph)}^2  } \leq \frac{\norm{\ph}_\infty^2}{N}.$$
(ii)~$\calM$-terms. Following \cite{v14}, we have
\begin{align*}
\E\left[\left(\int_0^T p_{t-}^N d\calM_t\right)^2\right]&=\E\left[ \sum_{j=1}^{B_T} \p{1-\tfrac{1}{N}}^{2j-2} (\Delta\calM_{\tau_{j}})^2\right]\\
&\leq\sum_{j=1}^{\infty} \p{1-\tfrac{1}{N}}^{2j-2} \E\left[ (\Delta\calM_{\tau_{j}})^2\right].
\end{align*}
For any $j\geq1$, there exists a unique couple $(n,k)$  with $n \in \set{1, \ldots , N}$ and $k\geq 1$ such that $\Delta\calM_{\tau_{j}}=\tfrac{1}{N}\Delta\calM^{n,k}_{\tau_{n,k}}$.
Then, by (\ref{aicj}) and by construction of the particle system, we have
\begin{align*}
&\E\left[ (\Delta\calM^{n,k}_{\tau_{n,k}})^2\right]\\
&\quad=\left(1-\frac{1}{N}\right)^2\ \E\left[\p{Q^{T-\tau_{n,k}}(\ph)(X_{\tau_{n,k}}^n)- \frac{1}{N-1} \sum_{m \neq n} Q^{T-\tau_{n,k}}(\ph)(X_{\tau_{n,k}}^m)}^2\right]\\
&\quad=\left(1-\frac{1}{N}\right)^2\ \E\left[\p{Q^{T-\tau_{n,k}}(\ph)(X_{\tau_{n,k}}^n)-\E\left[\left.Q^{T-\tau_{n,k}}(\ph)(X_{\tau_{n,k}}^n)\right|\calF_{\tau_{n,k}^-}\right]}^2\right]
\end{align*}
Again, it becomes clear that
$$\E\left[ (\Delta\calM^{n,k}_{\tau_{n,k}})^2\right]\leq(1-\tfrac{1}{N})^2\|\varphi\|^2_\infty,$$
so that 
\begin{align*}
\E\left[\left(\int_0^T p_{t-}^N d\calM_t\right)^2\right]&\leq \frac{\|\varphi\|^2_\infty}{N^2} \sum_{j=1}^\infty \p{1-\tfrac{1}{N}}^{2j}\\
&\leq \frac{\|\varphi\|^2_\infty}{N},
\end{align*}
the last inequality coming from the fact that $(1-x)^2/(1-(1-x)^2)\leq 1/x$ for any $x\in(0,1]$.\medskip

(iii)~$\M$-terms. By definition and orthogonality of the martingales $t\mapsto \M^{n}_t$, $1\leq n\leq N$, we have
$$\E\left[\left(\int_0^T p_{t-}^N d\M_t\right)^2\right]\leq\E\left[\left(\int_0^T d\M_t\right)^2 \right]=\E\left[\left(\M_T\right)^2 \right]=\frac{1}{N^2}\sum_{n=1}^N\E\left[\left(\M_T^n\right)^2 \right].$$
Next, keeping in mind that $k\leq B^n_t$ iff $t\geq \tau_{n,k}$, $k\leq B^n_t-1$ iff $t\geq \tau_{n,k-1}$, we are led to
\begin{align*}
\E\left[\p{\M^n_T}^2\right] &= \E\left[\sum_{k=0}^{B^n_T} \p{\M^{n,k}_T}^2\right]\\
  &= \E\left[  \sum_{k=0}^{B^n_T-1} \p{Q^{T-\tau_{n,k}}(\ph)\p{X^n_{\tau_{n,k}}}}^2 \right] \\
     &\quad + \E\left[ \p{\ph(X^n_{T}) -  Q^{T-\tau_{n,B^n_T}}(\ph)\p{X^n_{\tau_{n,B^n_T}}} }^2 \right] \\    
  &\leq  \norm{\ph}_\infty^2 \E\left[  \sum_{k=0}^{B^n_T-1} Q^{T-\tau_{n,k}}(\un_F)\p{X^n_{\tau_{n,k}}} \right] + 4 \norm{\ph}_\infty^2,
\end{align*}
and applying~\eqref{eq:Mbb} with the test function $\un_F$ gives
$$ \E \left[\p{\M^n_T}^2\right] \leq 5 \norm{\ph}_\infty^2 .$$
\end{proof}

In particular, $\gamma^N_T(\ph)$ converges in probability towards $\gamma_T(\ph)$ for any bounded $\ph$ when $N$ goes to infinity.

%----------------------------------------------------------------------
\subsection{Predictable quadratic variation}
%----------------------------------------------------------------------

Under Assumption~(CC), it is possible to compute the predictable quadratic variation of the martingale of interest. 

\begin{Lem}\label{cj} Recall that $M_t\eqdef    \gamma_t^N(Q^{T-t}(\ph))$ where $t \in [0,T]$.
If Assumptions~(SK) holds true and if $\ph$ satisfies Assumption~(CC), we have
 \begin{align*}
  N \d \bracket{M,M}_t =&%\var_{\gamma_0^N}{\ph} + \\
  \p{p^N_{t}}^2 \eta^{N}_{t} \p{ \Gamma_{T-t}(\ph)} \d t \\
  & + \p{p^N_{t}}^2  \left(\frac{1}{N}\sum_{n=1}^N  \lambda(X_t^n)\V_{\tfrac{1}{N-1}\sum_{m\neq n}\delta_{X_t^m}}\p{ Q^{T-t}(\ph) }\right) \d t.  %\\
  %\qquad + 
 \end{align*}

\end{Lem}
\begin{proof}
Let us recall Equation~(\ref{eq:quad}) :
$$\d \bracket{ M, M}_t = \frac{( p^N_{t-}  )^2}{N^2}\sum_{n=1}^N \sum_{k=0}^{B^n_t} \p{   \d  \bracket{ \M^{n,k},\M^{n,k} }_t +\d \bracket{ \calM^{n,k},\calM^{n,k} }_t }.$$
 (i)~$\M$-terms. By Assumption~(CC) and through a direct application of Doob's optional sampling theorem, 
  \[
  t \mapsto \p{\M^{n,k}_t}^2 - \int_{\tau_{n,k}}^{t \wedge \tau_{n,k+1}} \Gamma_{T-s}(\ph)\p{X^{n}_s} \d s
 \]
is a bounded martingale. Summing over $k$ and over $n$ yields the first term.\medskip

(ii)~$\calM$-terms. Note that the martingale $\calM^{n,k}$ is piecewise constant with a single jump at $\tau_{n,k}$. Hence the application of Lemma \ref{albcios} with $\tau=\tau_{n,k}$ and
$$U=\p{\calM^{n,k}_{\tau_{n,k}}}^2 -  \E\left[\left.  \p{\calM^{n,k}_{\tau_{n,k}} }^2  \right| \calF_{\tau_{n,k}^-} \right]$$
ensures that  
\[
U \un_{t \geq \tau_{n,k}}=\p{\calM^{n,k}_t}^2 -  \E\left[\left.  \p{\calM^{n,k}_{\tau_{n,k}} }^2  \right| \calF_{\tau_{n,k}^-} \right]\un_{t \geq \tau_{n,k}}
\]
defines a bounded c\`adl\`ag martingale. Moreover, by construction of the branching rule, we have
$$   \E\left[\left.\p{\calM^{n,k}_{\tau_{n,k}}}^2 \right| \calF_{\tau_{n,k}^-} \right] = V^n_{\tau_{n,k}^-},$$
where 
%\begin{align*}
$$V^n_{t}\ \eqdef\   % \E \b{  \p{\calM^{n,k}_{t\tau^{n,k}}}^2  | \calF_{\tau^{n,k}_-} } \\
 \p{1-\tfrac{1}{N}}^2\V_{\tfrac{1}{N-1}\sum_{m\neq n}\delta_{X^m_{t}}} \p{Q^{T-t}(\ph)}.
%%= \V_{\tfrac{1}{N}\sum_{m =1}^N\delta_{X^m_{t}} - \tfrac{1}{N}\delta_{X^n_{t}}} \p{Q^{T-t}(\ph)}.
$$
%\end{align*}
Therefore, 
$$t\mapsto \p{\calM^{n,k}_t}^2-V^n_{\tau_{n,k}^-}\un_{t \geq \tau_{n,k}}$$
is a bounded martingale. Moreover, Assumption (SK) ensures that
$$t\mapsto\widetilde{\calM}_t^{n,k}\ \eqdef\   \un_{t \geq \tau_{n,k}}-\int_{\tau_{n,k-1}}^{t\wedge\tau_{n,k}}\lambda(X_s^n) \d s$$
is a bounded martingale, and so is
$$t\mapsto \p{\calM^{n,k}_t}^2-\int_0^tV^n_{s^-}d\widetilde{\calM}_s^{n,k}-\int_{\tau_{n,k-1}}^{t\wedge\tau_{n,k}}V^n_{s}\lambda(X_s^n) \d s.$$
As a consequence,
$$\bracket{\calM^{n,k},\calM^{n,k}}_t = \int_{\tau_{n,k-1}}^{t\wedge\tau_{n,k}}V^n_{t} \lambda\p{X^n_s} \d s,$$
and we just have to sum over $k$ to conclude.
\end{proof}

% Define a compensated increasing process:
%  \begin{equation}
%   \d V^N_z\eqdef   \frac{( p^N_{z-}  )^2}{N^2}\sum_{n=1}^N\p{ \Gamma_z(\ph_z,\ph_z)(X^n_z)  \d z   + \var_{\frac{N}{N-1}\eta^{N}_{z'} -  \frac{1}{N} \delta_{.}}(\ph_z) \lambda_z(X^n_z) \d z},  
%  \end{equation}
% which reads:
% \begin{equation}
%  N V^N_{z_0} = \int_0^{z_0}  ( p^N_{z-}  )^2 \eta^N_{z} \p{\Gamma_{z}(\ph_{z},\ph_{z}) + \var_{\frac{N}{N-1}\eta^{N}_{z} -  \frac{1}{N} \delta_{.}}(\ph_z) \lambda_z            } \d z.
% \end{equation}

%----------------------------------------------------------------------
%----------------------------------------------------------------------

\subsection{The asymptotic variance and its simplification}
We first calculate the predictable quadratic variation in the many particles limit.

\begin{Lem}\label{almchi}
Under Assumption~(SK), if $\ph$ satisfies Assumption~(CC), then for any $ t \leq T$
\begin{align*}
 N\bracket{M,M}_t \xrightarrow[{N \to\infty}]{\L^1(\P)} \int_0^{t} \b{ \eta_{s} \p{\Gamma_{T-s}(\ph)} + \V_{\eta_{s}}(Q^{T-s}(\ph))     \eta_s(\lambda) }   p_{s}^2 \d s.
\end{align*}
\end{Lem}
\begin{proof}
 According to assumptions, all the quantities at stake are uniformly bounded by
 $$ \int_0^t \norm{\Gamma_{T-s}(\ph)}_\infty \d s+ \norm{\ph}^2_\infty\norm{\lambda}_\infty t< \infty.$$ 
Therefore, by dominated convergence, it is sufficient to prove that for any $s\in [0,t]$,
 $$A_N\eqdef   \p{p_{s}^N}^2\  \eta^N_{s} \p{\Gamma_{T-s}(\ph)}$$
 and 
 \begin{align*}
 B_N\eqdef  & \p{p^N_{s}}^2  \left(\frac{1}{N}\sum_{n=1}^N  \lambda(X_s^n)\V_{\tfrac{1}{N-1}\sum_{m\neq n}\delta_{X^m}}\p{ Q^{T-s}(\ph) }\right)\\
 =&\p{p^N_{s}}^2 \eta^{N}_{s} \p{  \lambda  \V_{\tfrac{N}{N-1}(\eta^{N}_{s} - \tfrac{1}{N} \delta_{.})} \p{ Q^{T-s}(\ph) } }
\end{align*}
 both converge in $\L^1(\P)$, or equivalently in distribution, towards their deterministic limits.\medskip
 
Since $\eta^N_s = \gamma^N_s /\gamma^N_s(\un_F)$, Proposition \ref{pro:estimate} ensures that for any bounded function $\psi$, one has 
$$\eta_s^N(\psi)\xrightarrow[N\to\infty]{\P}\eta_s(\psi).$$
In particular, since $p^N_{s}=\gamma^N_s(\un_F)$, it is clear that
$$A_N\xrightarrow[N\to\infty]{\P}p_{s}^2\  \eta_{s} \p{\Gamma_{T-s}(\ph)}.$$
The term $B_N$ admits the alternative formulation
\begin{align*}
B_N=&\p{p^N_{s}}^2\ \eta_s^N\p{\lambda\times \left(\tfrac{N}{N-1}\eta^{N}_{s}(Q^{T-s}(\ph)^2)-\tfrac{1}{N-1}Q^{T-s}(\ph)^2\right)}\\
&-\p{p^N_{s}}^2\ \eta_s^N\p{\lambda\times \left(\tfrac{N}{N-1}\eta^{N}_{s}(Q^{T-s}(\ph))-\tfrac{1}{N-1}Q^{T-s}(\ph)\right)^2}.
\end{align*}
By expanding again the latter, it turns out that $B_N$ might be expressed as a continuous functions of $\eta_s^N(\lambda)$, $\eta^{N}_{s}(Q^{T-s}(\ph))$, etc. Finally, since convergence in probability is stable by continuous mapping, Proposition~\ref{pro:estimate} allows us to conclude that
$$B_N\xrightarrow[N\to\infty]{\P}p_{s}^2\ \V_{\eta_{s}}(Q^{T-s}(\ph))     \eta_s(\lambda),$$
and the proof is complete.
\end{proof}

The variance formula will then be denoted as follows.
\begin{Def} Assume~(SK) and suppose that $\ph$ satisfies~(CC). Then, for any $t \geq 0$, we define the asymptotic variance of $M_t$ by
\begin{align}\nonumber
\sigma_t^2(\ph) \eqdef& \V_{\eta_0}(Q^{T-t}(\ph))\\ 
&+  \int_0^{t} \b{ \eta_{s} \p{\Gamma_{T-s}(\ph)}
+ \V_{\eta_{s}}(Q^{T-s}(\ph))     \eta_s(\lambda) }   p_{s}^2 \d s.\label{eq:var_1}
 \end{align}
\end{Def}

One can then simplify the variance at final time $T$ as follows. 
\begin{Lem}
Assume~(SK) and suppose that $\ph$ satisfies~(CC). Then, for $t=T$, the variance defined by~\eqref{eq:var_1} satisfies
\begin{equation}\label{eq:var_2}
\sigma_T^2(\ph) = p^2_T \V_{\eta_T}(\ph) - p_T^2\ln(p_T) \, \eta_T(\ph)^2 - 2\int_0^T \V_{\eta_{t}}(Q^{T-t}(\ph)) p_t p'_t \d t.
\end{equation}
\end{Lem}
\begin{proof}
Since $\ph(\partial)=0$, by definition of $\Gamma$ and using Assumption~(CC) we have 
$$\E\left[\p{  Q^{T-t}(\ph)(X_t) }^2\un_{t\leq \tau_\partial} \right]=\E\left[\int_0^{t} \Gamma_{T-s}(\ph)(X_s)\un_{s\leq \tau_\partial} \d s\right],$$
which amounts to say that
$$\gamma_t\p{ \p{Q^{T-t}(\ph)}^2 }=\int_0^{t} \gamma_s(\Gamma_{T-s}(\ph)) \d s,$$
and we are led to
\[
\frac{\d }{\d t} \gamma_t\p{ \p{Q^{T-t}(\ph)}^2 } = \gamma_t \p{\Gamma_{T-t}(\ph)}.
\]
In the same manner, by Assumption~(SK), we also have
$$p_t=\gamma_t(\un_F)=\P(X_t\neq\partial)=1-\E\left[\int_0^{t} \lambda(X_s)\un_{s\leq \tau_\partial} \d s\right],$$
hence
$$p'_t\ \eqdef\  \frac{\d}{\d t} p_t =\frac{\d}{\d t}\gamma_t(\un_F)=-\E\left[\lambda(X_t)\un_{t\leq \tau_\partial}\right]=-\gamma_t(\lambda).$$
We have now the tools to simplify~(\ref{eq:var_1}). The factor $\eta_s(\lambda)=\gamma_s(\lambda)/p_s$ can be replaced by $-p_s'/p_s$. 
For the other term inside the integral, recall that for any bounded function $\ph$ and any $t$, we have $\gamma_0(Q^T\ph)=\gamma_t(Q^{T-t}\ph)=\gamma_T(\ph)$, and
%\begin{align*}
$$\gamma_{t}(\ph^2) = p_t \V_{\eta_t}(\ph) + \frac{\gamma_t(\ph)^2}{p_t},$$
%\end{align*}
so that
\begin{align*}
\int_0^T  &\eta_t \p{\Gamma_{T-t}(\ph)}p_t^2 \ dt \\
&=\int_0^T p_t \frac{\d }{\d t} \gamma_t\p{ \p{Q^{T-t}(\ph)}^2 } \ dt  \\
&= p_T^2\eta_T(\varphi^2) -  \eta_0(\varphi^2)-  \int_0^T p_t' \left(p_t \V_{\eta_t}(Q^{T-t}\ph) + \frac{\gamma_t(Q^{T-t}\ph)^2}{p_t}\right)  \d t \\
& =p^2_T \V_{\eta_T}(\ph) - \V_{\eta_0}(\ph) - \ln(p_t)  \, \gamma_T(\ph)^2 - \int_0^T \V_{\eta_{t}}(Q^{T-t}(\ph))p_t  p'_t \d t.
\end{align*}
Finally we get
\begin{equation*}
 \sigma_T^2(\ph) = p^2_T \V_{\eta_T}(\ph) - p_T^2\ln(p_T) \, \eta_T(\ph)^2 - 2\int_0^T \V_{\eta_{t}}(Q^{T-t}(\ph)) p_t p'_t  \d t.
\end{equation*}
\end{proof}

\subsection{Martingale Central Limit Theorem}
The following result is an adaptation of Theorem $1.4$ page 339 in~\cite{ek86} to our specific context. The main difference is about the initial condition.

\begin{The}\label{aeichj}
Let $t \mapsto Z_t^N$ denote a sequence of c\`adl\`ag processes indexed by $N \geq 1$, which may be defined on different probability spaces. Suppose that $t \mapsto Z_t^N-Z^N_0$ are local square integrable martingales, and assume moreover that
\begin{enumerate}[(a)]
\item $Z_0^N \xrightarrow[N \to + \infty]{\calD} \mu_0$, where $\mu_0$ is a given probability on $\R$,
% ; ie such that
% \[
% z \mapsto (M_z^N )^2 - V^N_z 
%  \]
% is a local martingale. 
\item the following limit holds
\[
\lim_{N \to +\infty} \E\left[\sup_{t \in [0,T]}\abs{\Delta Z^N_t }^2\right] =  0,
%\lim_{N \to +\infty} \E(\sup_{z \in [0,1]}\abs{\Delta \bracket{M^N,N^N}_t } ) = 0. 
 \]
\item for each $N$, the predictable quadratic variation $\bracket{Z^N,Z^N}_t$ is continuous,
\item there exists a continuous and increasing deterministic function $t \mapsto v(t)$ such that, for all $t\in[0,T]$,
 \[
  \bracket{Z^N,Z^N}_t \xrightarrow[N \to + \infty]{\calD} v(t).
 \]
 \end{enumerate}
Then $(Z_t^N)_{t \in [0,T]}$ converges in law (under the Skorokhod topology) towards $(Z_t)_{t \in [0,T]}$, where  $Z_0 \sim \mu_0$ and  $(Z_t-Z_0)_{t \in [0,T]}$ is a Gaussian process, independent of $Z_0$, with independent increments and variance function $v(t)$.  
\end{The}

\begin{proof}
The proof is a slight extension of Theorem $1.4$ in~\cite{ek86}, where only the special case $Z_0^N=0$ is stated. See also Section~$5$, Chapter~$7$ of~\cite{js03}, in which, again, the case of a general initial condition is left to the reader. \medskip

Let us 
% consider the martingale
% $
% t \mapsto Z^N_t-Z^N_0,
% $
fix $\psi \in C_b(\R)$, and consider the $\P$-absolutely continuous probability defined by
$$ \P_\psi = \frac{1}{\E \b{  e^{\psi(Z^N_0)} }} e^{\psi(Z^N_0)}\ \P.$$
Note that, in this notation, the dependency of $\P_\psi$ with respect to $N$ has been dropped for simplicity. We claim that, under $\P_\psi$, all the assumptions of the theorem still hold for $t \mapsto Z^N_t-Z^N_0$.\medskip

Indeed, under $\P_\psi$, since $\psi$ is bounded, the integrability property of $t \mapsto Z^N_t-Z^N_0$ and $(b)$ are not modified. Moreover, under $\P_\psi$, since $Z^N_0$ is measurable with respect to the initial $\sigma$-field, martingale properties are not modified so that $t \mapsto Z^N_t-Z^N_0$ is still a local martingale, and the predictable quadratic variation is not modified (predictability is insensitive to change of probabilities). Again, since $\psi$ is bounded, convergence in law towards a constant is unchanged, so that $(d)$ still holds. Besides, $(c)$ still holds by absolutely continuity. \medskip

As a consequence, for any continuous functional $F$ on the Skorokhod space of c\`adl\`ag  paths, Theorem $1.4$ in~\cite{ek86} ensures that
\begin{align*}
 \E \b{ e^{\psi(Z^N_0)} F(Z^N_t-Z^N_0, \, t \geq 0)} &= \E_\psi\p{ F(Z^N_t-Z^N_0, \, t \geq 0)} \E \b{e^{\psi(Z^N_0)} } \\
 & \xrightarrow[N \to +\infty]{} \E \b{ F(Z_t-Z_0 , \, t \geq 0 )  }\mu_0(e^\psi),
\end{align*}
which is precisely the desired result.
\end{proof}

\begin{Rem}
In other words, the limit Gaussian process $(Z_t)_{t \in [0,T]}$ is solution of the stochastic differential equation
$$\left\{\begin{array}{ll}
Z_0&\sim \mu_0\\
dZ_t&=\sqrt{v(t)}\ dWt
\end{array}\right.$$
where $(W_t)_{t \in [0,T]}$ is a standard Brownian motion.
\end{Rem}

\begin{Pro}\label{lazicj} For any $\ph$ satisfying Assumption~(CC), the martingale $(Z_t^N)_{0\leq t\leq T}$ defined by
 $$Z_t^N=\sqrt{N}\left(\gamma_t^N(Q^{T-t}(\ph))-\gamma_0(Q^{T-t}(\ph))\right)$$
converges in law towards a Gaussian process $(Z_t)_{t \in [0,T]}$ with independent increments, initial distribution ${\cal N}(0,\V_{\eta_0}(Q^T(\ph)))$ and variance function $\sigma_t^2(\ph)$ defined by~\eqref{eq:var_1}-\eqref{eq:var_2}.
 \end{Pro}
\begin{proof}
%%***REGLER LES EMBROUILLES AVEC LES $\sigma(\varphi)$*** ---->  Done.
We just have to check that the assumptions of Theorem \ref{aeichj} are satisfied in our framework.
\begin{enumerate}[(a)]
\item Recall that $(X_0^1,\ldots,X_0^N)$ are i.i.d.~with law $\eta_0=\gamma_0$. Since $\ph$ is assumed bounded and $\eta_0^N=\gamma_0^N$, the central limit theorem ensures the asymptotic normality of $Z_0^N$ with asymptotic variance $\sigma_0^2=\V_{\eta_0}(Q^T(\ph))$.
\item  Since $\ph$ is bounded, this property is satisfied because the sets ${\cal D}^1,\ldots,{\cal D}^N$ of discontinuities of the particles have almost surely an empty intersection, as stated in Lemmas \ref{lem:jumps} and \ref{jump2}.
\item According to Lemma \ref{cj}, the predictable quadratic variation of the martingale is
\begin{align*}
 N \bracket{M,M}_t=&%\var_{\gamma_0^N}{\ph} + \\
 \V_{\eta_0^N}(\ph)+\int_0^t \p{p^N_{s}}^2 \eta^{N}_{s} \p{ \Gamma_{T-s}(\ph)} \d s \\
  +&\int_0^t  \p{p^N_{s}}^2  \left(\frac{1}{N}\sum_{n=1}^N  \lambda(X_s^n)\V_{\tfrac{1}{N-1}\sum_{m\neq n}\delta_{X_s^m}}\p{ Q^{T-s}(\ph) }\right) \d s,  %\\
  %\qquad + 
 \end{align*}
 which is clearly continuous with $t$ since $\ph$ and $\lambda$ are both bounded.
 \item This last point is a consequence of Lemma \ref{almchi}.
 \end{enumerate} 
 \end{proof}

If we marginalize on the final time, we obtain that, for any $\ph$ satisfying Assumption~(CC),
$$\sqrt{N}\left(\gamma_T^N(\ph)-\gamma_T(\ph)\right)\xrightarrow[N\to\infty]{\cal D}{\cal N}(0,\sigma_T^2(\ph)).$$
We can extend this latter result to any  function $\ph$ in the $\norm{\cdot}_\infty$-closure of the set of functions satisfying~(CC), and thus establish Theorem \ref{gamma}.

\begin{Cor}\label{corbis}
For any $\ph$ belonging to the $\norm{\cdot}_\infty$-closure of the set of functions satisfying~(CC), one has 
$$\sqrt{N}\left(\gamma_T^N(\ph)-\gamma_T(\ph)\right)\xrightarrow[N\to\infty]{\cal D}{\cal N}(0,\sigma_T^2(\ph)).$$
 \end{Cor}

\begin{proof}
Let us denote by $ \calD$ the set of functions satisfying~(CC), and by $\overline{\cal D}$ its $\norm{\cdot}_\infty$-closure.
We will use  the simplified version of the asymptotic variance, namely~\eqref{eq:var_2}. 
Let us denote by $\Phi$ any bounded Lipschitz function, $G$ a centered Gaussian variable with variance $\sigma_T^2(\varphi)$ for an arbitrary  function $\varphi\in\overline{\cal D}$. \medskip

For any  $\varepsilon>0$, we can find $\varphi_\varepsilon$ in $ \calD$ such that $\|\varphi-\varphi_\varepsilon\|_\infty\leq\varepsilon$. We can also assume that $\gamma_T(\varphi_\varepsilon)=\gamma_T(\varphi)$. Note that we can also choose $\varphi_\varepsilon$ such that $|\sigma_T^2(\varphi_\varepsilon)-\sigma_T^2(\varphi)|\leq \varepsilon$. Indeed, it is easy to check by dominated convergence that $\varphi\mapsto\sigma_T^2(\varphi)$ is continuous for the norm $\|\cdot \|_\infty$. Hence, let us denote by $G_\varepsilon$ a centered Gaussian variable with variance $\sigma_T^2(\varphi_\varepsilon)$.\medskip

Then we may write
\begin{align*}
|\E[\Phi(\sqrt{N}&(\gamma_T^N(\varphi)-\gamma_T(\varphi))]-\E[\Phi(G)]|\\
\leq&\ \E[|\Phi(\sqrt{N} (\gamma_T^N(\varphi)-\gamma_T(\varphi)))-\Phi(\sqrt{N} ( \gamma_T^N(\varphi_\varepsilon)-\gamma_T(\varphi)))|]\\
&+|\E[\Phi(\sqrt{N} ( \gamma_T^N(\varphi_\varepsilon)-\gamma_T(\varphi)))]-\E[\Phi(G_\varepsilon)]|\\
&+|\E[\Phi(G_\varepsilon)]-\E[\Phi(G)]|.
\end{align*}

For the first term, by Proposition \ref{pro:estimate}, Jensen's inequality and remembering that $\gamma_T(\ph-\varphi_\varepsilon)=0$, we have
$$\E[|\Phi(\sqrt{N} (\gamma_T^N(\varphi)-\gamma_T(\varphi)))-\Phi(\sqrt{N} ( \gamma_T^N(\varphi_\varepsilon)-\gamma_T(\varphi)))|]\leq\sqrt{7}\|\Phi\|_{\rm Lip}\|\varphi-\varphi_\varepsilon\|_\infty.$$
Hence, for any given $\delta>0$, we can choose $\varepsilon$ such that this first term is less than $\delta$. Clearly, the same property holds for the third term as well. Besides, since $\ph_\varepsilon$ is in $\calD$, for $N$ large enough, the second term can also be made less than $\delta$ by Corollary \ref{lazicj}. As this result holds for any bounded Lipschitz function $\Phi$, we conclude using the Portmanteau theorem.

\end{proof}

 \begin{Rem}
  This corollary is particularly useful in practice: to obtain the CLT associated with any observable $\ph$, it is sufficient to check Assumption~(CC) for appropriately regularized functions. 
 \end{Rem}
 
%----------------------------------------------------------------------
%----------------------------------------------------------------------
\section{Technical results}\label{appendix}
%----------------------------------------------------------------------
%----------------------------------------------------------------------

\subsection{General setting and construction of the particle system}\label{sec:def_ips}

As mentioned in Section~\ref{sec:main}, we have assumed that the underlying Markov process at stake is c\`adl\`ag with Polish state space $F$. Actually, our arguments do not involve the specific topology of $F$. \medskip

A possible general setting is the following. $F$ is assumed to be \emph{standard Borel}, which means that it is associated with a $\sigma$-field which is the Borel $\sigma$-field of some \emph{unspecified} Polish topology. Then, we assume that the Markov process is \emph{constructible} in the sense that there is a \emph{jointly measurable} mapping
$ \chi: (x,t,u) \in F \times \R^+ \times [0,1] \mapsto F \cup \set{\partial} $
such that the underlying Markov process is constructed from any initial condition $X_0$ by setting
$$ X_t = \chi(X_0,t ,U),$$
where $U$ is uniform and independent of $X_0$. If $F$ is Polish and $X$ is c\`adl\`ag, then $X$ is automatically constructible, since the Skorokhod space of c\`adl\`ag paths with values in a Polish space $F$ is itself Polish, ensuring in particular the existence of regular conditional distributions. \medskip
 
In this general setting, the minimal right-continuous filtration generated by $X$ is defined as the minimal \emph{right-continuous} filtration  $(\calF_t^X)_{t \geq 0}$ making $X$ progressively measurable.
% , and such that the Markov property holds true: %%$\mu_0 = \calL(X_0)$, 
%  $ \forall \mu_0= \calL(X_0), \, \forall A \in \calB(F), \, \forall s,t \geq 0$,
% \[
%  \P\p{ X_{t+s} \in A | \calF^X_t} =   \P\p{ X_{t+s} \in A | X_t } = Q^s(X_t, A).
% \]

% , is a given contained in $\sigma(X_0,U)$ making $X$ $\calF^X$-progressively measurable. Recall that progressive measurability is equivalent to $X_\tau \in \calF^X_\tau$ for any $\calF^X$-stopping time $\tau$. If $X$ is c\`adl\`ag, it is easy to check that $\calF^X_t = \bigcap_{\eps > 0} \sigma(X_{s}, \, s \leq t + \eps)$ is the standard right-continuous extension of the natural filtration. \medskip

\medskip

Next, the particle system can be rigorously constructed without measurability issue using $\chi$ by setting $X^{n,k}_t  = \chi(t-\tau_{n,k}, X^{n,k}_{\tau_{n,k}},U_{n,k})$ for $ \tau_{n,k}\leq t \leq \tau_{n,k+1}$, where $U_{n,k}$ are i.i.d, $1 \leq n \leq N$ and $k \geq 0$. The filtration of the particle system $\calF_t$ at time $t$ is then generated by the events of the form
$$ A_h \cap \set{\tau_{n,k} + h \leq t}$$
where $A_h \in \calF_h^{X^{n,k}}$, $h \geq 0$, $1 \leq n \leq N$, and $k \geq 0$.

\subsection{The Feller case}\label{sec:fellerannex}
%In practice, for the usual class of Feller-Dynkin processes - such as regular diffusions or Levy proceses -, checking the continuity of $x \mapsto L(\one_F)(x)$ or $(x,t) \mapsto \Gamma_{t}(\ph)(x)$, where the operator $\partial_t+L$ may be defined in a very formal sense, is enough in order to check Assumptions~(SK) or Assumption~(CC). 
The general results on Feller processes mentioned in this section can be found in \cite{RogWil_vol1}, Chapter~III.\medskip

In this section, we assume that the Polish state space $K \eqdef F \cup \set{\partial}$ is \emph{compact}. The compactification of locally compact spaces is standard, see for example ~\cite{RogWil_vol1}. We denote by $C(K)$ the space of continuous functions on $K$, endowed with the uniform norm. Let us recall that the semi-group $(Q^t)_{t \geq 0}$ is called Feller-Dynkin, or in short Feller, if $t \mapsto Q^t(\ph)$ is continuous in $C(K)$ for any $\ph \in C(K)$. \medskip

In practice, the following weaker conditions are sufficient to check that a Markov semi-group is Feller: (i) $\lim_{t \to 0} Q^t(\ph)(x) = \ph(x)$ for each $x \in K$ and each $\ph \in C(K)$; (ii) for each $t \geq 0$ and each $\ph \in C(K)$,  $Q^t(\ph) \in C(K)$. \medskip

Any Markov process with a Feller semi-group of probability transitions has a \cadlag modification. Among the many nice properties of the latter, we will use the following ones:
\begin{itemize}
\item first, the minimal filtration generated by the process is right-continuous (Blumenthal $0-1$ law), so that the process is automatically Markov with respect to the minimal right-continuous filtration it generates. \item second, the generator $L$ can be defined as the $C(K)$-infinitesimal generator of the semigroup $(Q^t)_{t \geq 0}$. The - dense - domain $\calD(L) \subset C(K)$ is defined as the set of functions $\ph$ such that 
$$L(\ph) = \lim_h \frac{Q^h\ph - \ph}{h}\in C(K),$$
whenever the limit exists. Uniformity in the latter definition yields $L Q^t = Q^t L$ so that if $\ph \in \calD(L)$, then for any initial distribution of $X_0$, the process $t \mapsto \ph(X_t) - \int_0^t L(\ph)(X_s) \d s$ is a martingale with respect to the natural filtration of $X$.
\end{itemize}

The following lemma may be useful in practice to check that a given function $\ph$ belongs to $\calD(L)$. %Note that if $L(\ph)$ is continuous and bounded and b
\begin{Lem}\label{lem:domain}
 Let $(X_t)_{t \geq 0}$ be Feller. If the pointwise limit
 $$ L(\ph)(x)  = \lim_{h \downarrow 0} \frac{Q^h(\ph)(x)-\ph(x)}{h} $$ 
 is continuous with respect to $x \in K$, then $\ph \in \calD(L)$, the domain of the $C(K)$-infinitesimal generator.
\end{Lem}
\begin{proof}
 It is a consequence of the Dynkin-Reuter lemma (Lemma~$4.17$, Chapter~III of~\cite{RogWil_vol1}), in the context of Hille-Yosida semigroup theory. Indeed, if $L(\ph) = \ph \in C(K)$ where $L(\ph)$ is defined in a pointwise sense, then by construction $L(\ph)(x_{\rm max}) \leq 0$ if $\ph(x_{\rm max}) = \sup \ph$, so that $\ph \leq 0$. The same reasoning leads to $-\ph \leq 0$, and consequently $\ph = 0$.
\end{proof}

We can then deduce easily the following result.
\begin{Pro}
 Let $(X_t)_{t \geq 0}$ be Feller. If there exists a continuous function $\lambda$ such that
for any $x\in F$
$$ \lambda(x)  = \lim_{h \downarrow 0} \frac{\P_x(X_h=\partial)}{h},$$ 
then Assumption~(SK) is satisfied.
\end{Pro}
A criterion for~(CC) is the following:
\begin{Pro}\label{lem:comp_fell}
 Let $(X_t)_{t \geq 0}$ be Feller. Assume that $\ph \in \calD(L)$ and that $\p{Q^{t}(\ph)}^2 \in \calD(L)$ for any $t\in(0,T)$. Set 
 $$\Gamma(\ph,\ph) =L(\ph^2) - 2 \ph L\ph\hspace{1cm}\mbox{and}\hspace{1cm}\Gamma_t(\ph) = \Gamma\p{Q^{t}(\ph),Q^{t}(\ph)},$$
 and suppose that $\sup_{(x,t) \in K \times (0,T)} \Gamma_t(\ph)< + \infty$. 
 Then $\ph$ satisfies Assumption~(CC).
\end{Pro}
\begin{proof}
According to Lemma~\ref{lem:gen}, Assumption~(CC) will be a consequence of the formula
\begin{equation}\label{eq:lim_fell}
 \frac{\d}{\d t} \eta_0 Q^t\p{\p{Q^{T-t}(\ph)}^2 } = \eta_0 Q^t \p{ \Gamma\p{Q^{T-t}(\ph),Q^{T-t}(\ph)} }
\end{equation}
for each initial $\eta_0 = \calL(X_0)$ and each $t \in (0,T)$. \medskip
Then consider 
\begin{align*}
\frac{\d}{\d t} \eta_0 Q^t\p{\p{Q^{T-t}(\ph)}^2 }  = 
& \quad \lim_{h \to 0}  \, \, 
\frac{1}{h}\eta_0 Q^{t+h}\p{\p{Q^{T-t-h}(\ph)}^2 - \p{Q^{T-t}(\ph)}^2 } \\&\hspace{3em}+
\frac{1}{h}\eta_0 \p{Q^{t+h}-Q^t}\p{\p{Q^{T-t}(\ph)}^2 },
\end{align*}
and remark that in the above, the first term of the right hand side converges to 
$$- 2 \eta_0 Q^{t}\p{Q^{T-t}(\ph) L \b{Q^{T-t}(\ph)}}$$
by the uniform convergence in the Feller case of 
$$\lim_{h \to 0} \frac{Q^{T-t-h}(\ph)-Q^{T-t}(\ph)}{h} = -L(Q^{T-t}(\ph))$$ 
when $ \ph \in \calD(L)$. Besides, the second term of the right hand side goes to  
$$\eta_0 Q^{t}\p{ L \b{\p{Q^{T-t}(\ph)}^2}}$$ since by assumption $\p{Q^{T-t}(\ph)}^2$ belongs to $\calD(L)$ for any fixed $t \in (0,T)$.
% 
% $$
% (\partial_t+L)\p{ \p{Q^{T-t}(\ph)}^2} = \Gamma\p{Q^{t}(\ph),Q^{t}(\ph)}
% $$
% where $ \partial_t+L $ is defined as an extended generator and $\p{Q^{T-t}(\ph)}^2$ belongs to its domain. By Lemma~\ref{eq:def_gen}, this is implied by - in fact equivalent to in the Feller case -
% \begin{equation}\label{eq:lim_fell}
%  \frac{\d}{\d t} \eta_0 Q^t\p{\p{Q^{T-t}(\ph)}^2 } =\eta_0 Q^t \p{ \Gamma\p{Q^{t}(\ph),Q^{t}(\ph)} }
% \end{equation}
% 
% 
% Remark that if $t \mapsto X_t \in K$ is Feller, so is $t \mapsto (X_t,t) \in K \times [0,T]$. Moreover,
% by assumption, $(x,t) \mapsto \Gamma\p{Q^{t}(\ph),Q^{t}(\ph)}$ is continuous on $K \times [0,T]$. According to Lemma~\ref{lem:domain} above, Assumptions~(CC) will be a consequence of the formula
% $$
% (\partial_t+L)\p{ \p{Q^{T-t}(\ph)}^2} = \Gamma\p{Q^{t}(\ph),Q^{t}(\ph)}
% $$
% where $ \partial_t+L $ is defined as an extended generator and $\p{Q^{T-t}(\ph)}^2$ belongs to its domain. By Lemma~\ref{eq:def_gen}, this is implied by - in fact equivalent to in the Feller case -
% \begin{equation}\label{eq:lim_fell}
%  \frac{\d}{\d t} \eta_0 Q^t\p{\p{Q^{T-t}(\ph)}^2 } =\eta_0 Q^t \p{ \Gamma\p{Q^{t}(\ph),Q^{t}(\ph)} }
% \end{equation}
% for each $\eta_0 = \calL(X_0)$ \medskip
% 
% Now remark that $Q^t(\ph)$ belong to $\calD(L)$ for each $t \in [0,T]$, so that by definition of the $C(K)$-infinitesimal generator, the $\lim_{h \to 0} \frac{Q^{T-t-h}(\ph)-Q^{T-t}(\ph)}{h} = L(Q^{T-t}(\ph))$ is uniform, which implies~\ref{eq:lim_fell} by a straightforward development.
\end{proof}

% Step~$1$. We prove that $\Gamma_t(\ph) = \Gamma\p{Q^{t}(\ph),Q^{t}(\ph)}$ in a pointwise or extended sense. Indeed This is a consequence of Lemma~ ?????????? and uniform convergence in the formula $\frac{\d}{\d t} Q^t = L(Q^t)$ when $\ph \in \calD{L}$. 
% 
%  Step~$2$.. Moreover, by the pointwise definition of $\Gamma_t$, $\Gamma_{T-t}(\ph)(x) = \p{\partial_t + L } \b{\p{Q^{T-t}(\ph)}^2 }$, where $\partial_t + L $ denotes the pointwise generator of $t \mapsto (X_t,t)$. By Lemma~\ref{lem:domain}, if $(x,t) \mapsto \Gamma_{T-t}(\ph)(x)$ is continuous, $\p{Q^{T-t}(\ph)}^2$ lies in the domain of $\partial_t + L$ in the $C_0(F\times[0,T])$-semigroup sense, and thus solves the associated martingale problem, which is exactly Assumption~(CC).
%\end{proof}

\subsection{Carré du champ}\label{sec:carre}
% Let us also recall first that formally, the ``carr\'e du champ'' operator may be defined for any test function $\ph$ and any $x\in F$ by
% \begin{equation*}\label{eq:def_cc}
%  \Gamma(\ph,\ph)(x)\eqdef  \lim_{h\downarrow 0}\frac{\V(\ph(X_h)|X_0=x)}{h},
% \end{equation*}
% whenever this limit makes sense. As a consequencee, $\Gamma(\ph,\ph)=L\ph^2-2\ph L\ph$ as soon as $L\ph$ and $L \ph^2$ can be defined in a pointwise sense. 

% As a consequence, one can check that the function $\Gamma_{t}$ in Assumption~(CC) is in fact equal to
% \[
%  \Gamma_{t}(\ph)(x) = \Gamma\p{Q^{t}(\ph),Q^{t}(\ph)}(x).
% \]
% 
% Thus, Assumption~(CC) amounts to require that either the mapping
% $$t\mapsto -\frac{\d}{\d t} \V( \ph(X_T) | X_t) \geq 0,$$
% or equivalently the time derivative of the predictable quadratic variation of the martingale $t \mapsto Q^{T-t}(\ph)(X_t)$, is almost surely upper-bounded by a time integrable deterministic function.
%  
% 
% \eqref{eq:def_cc_2}~is 
If the time-dependent functions $\psi(t,x)$ and $\psi^2(t,x)$ belong to the domain of the extended generator of $t \mapsto (t,X_t)$ as defined in Section \ref{maoezcj}, then considering the martingale 
$$C_t(\psi) \eqdef \psi(t,X_t) - \psi(0,X_0) - \int_0^t \p{\partial_s+L}(\psi)(s,X_s) \d s,$$ 
it can be checked that the process
\[
t\mapsto C^2_t(\psi) - \int_0^t \Gamma(\psi,\psi)(s,X_s) \d s
\]
is a martingale, where $\Gamma$ is the so-called ``carr\'e du champ'' operator defined by
$$\Gamma(\psi,\psi) = \p{\partial_t+L}( \psi^2 ) - 2 \psi \p{\partial_t+L}(\psi).$$
Note that, for a time-independent function $\ph$, one has
$$\Gamma(\ph,\ph) = L( \ph^2 ) - 2 \ph L(\ph).$$
The process 
$$t \mapsto \int_0^{t} \Gamma(\psi,\psi)(s,X_s) \d s$$ 
may be defined as the unique continuous increasing process appearing in the Doob-Meyer decomposition of the bounded sub-martingale $t \mapsto C^2_t(\psi)$. This is exactly the definition of the predictable quadratic variation of $t \mapsto C_t(\psi)$, so that 
$$  \frac{\d}{\d t} \bracket{C(\psi),C(\psi)}_t = \Gamma(\psi,\psi)(t,X_t).$$
In most cases, one has the compatibility formula - formally Leibniz chain rule with respect to time differentiation -
\begin{equation}\label{eq:comp}
\Gamma(\psi(t,\, .),\psi(t, \, .))(x) = \Gamma(\psi,\psi)(t,x),
\end{equation}
but some care is needed since the latter is only formal in general.
%It is true in the Feller case, as soon as $\psi(t,\, .)$ and $\psi^2(t, \, .)$ belong to the domain of the infinitesimal generator for each $t$, see above.
% If~\eqref{eq:def_gen} is true for both $\ph$ and $\ph$, the ``carr\'e du champ'' operator then satisfies the pointwise formula,
% \begin{align}\label{eq:def_cc}
%  \Gamma(\ph,\ph)(t,x) &\eqdef  \lim_{h\downarrow 0}\frac{\V(\ph(t+h,X_{t+h})|X_t=x)}{h} \geq 0 \nonumber \\
%   &  = \lim_{h\downarrow 0}\frac{ Q^h\p{\ph^2(t+h,.)}(x) - \p{Q^h(\ph(t+h,.))}^2(x) }{h}.
% \end{align}
In this approach, the function $\Gamma_{T-t}(\ph)$ in Assumption~(CC) is thus rigorously given by
\[
\Gamma_{T-t}(\ph)(x) = \Gamma \p{ Q^{T-\,.}(\ph),Q^{T- \, .}(\ph) } (t,x),
\]
which is simply the carr\'e du champ computed with the time-dependent test function $\psi(t,x) = Q^{T-t}(\ph)(x)$. In the Feller case, the compatibility condition~\eqref{eq:comp} is ensured by Proposition~\ref{lem:comp_fell}.\medskip

By definition, one can also see $\Gamma_{T-t}(\ph)(X_t)$ as the time derivative of the predictable quadratic variation of the martingale $t \mapsto \M_t(\ph) \eqdef Q^{T-t}(\ph)(X_t)$, which reads $$ \Gamma_{T-t}(\ph)(X_t) = \frac{\d}{\d t} \bracket{\M(\ph),\M(\ph)}_t.$$

Finally, $\Gamma_{T-t}(\ph)(X_t)$ can also  can be interpreted as the derivative of a variance as follows
\begin{align}\label{eq:def_cc_2}
 \E \b{\Gamma_{T-t}(\ph)(X_t)} & = -\frac{\d}{\d t} \E \b{\V(\ph(X_T)|X_t) } \geq 0. %\nonumber \\
% & = \lim_{h\downarrow 0}\frac{ Q^h\b{\p{Q^{t-h}(\ph)}^2}(x)    -  \p{Q^{t}(\ph)(x)}^2 }{h}.
\end{align}
Indeed, by Jensen's inequality, $h \mapsto \E\b{\V(\ph(X_T)|X_{t+h}) | X_t=x}$ is decreasing so that the limit in \eqref{eq:def_cc_2} always exists in $[0,+\infty]$.\medskip

\subsection{Proof of Proposition~\ref{pdmp}}\label{amoch}
We start with two simple standard remarks.
\begin{Lem}\label{lem:pdmp1}
If $\psi$ is a bounded function, then the process
\[
t \mapsto\sum_{s \leq t,\ s \in \set{T_1,T_2, \cdots, \tau_\partial}}  
\left(\psi(s,X_{s^-},X_s)  
- \int_{x' \in F \cup \set{ \partial }} \psi(s,X_{s^-},x') \frac{q(X_{s^-},\d x')}{\bar{q}(X_{s^-})}\right)
\]
is a martingale.
\end{Lem}
\begin{proof}
By construction, the random variables generating the jumps and the jump times of the PDMP are independent. As a consequence, Lemma~\ref{albcios} can be applied to each jump time, and the result follows by summing over these jump times. Then, one can see that the resulting local martingale is indeed a martingale by comparing it to a Poisson process with rate $\norm{\bar{q}}_{\infty}$.
\end{proof}

\begin{Lem}\label{lem:pdmp2}
If $\phi$ is a bounded function and if the process $s \mapsto \phi(s,X_{s})$ admits a version with a left limit denoted $\phi(s^{-},X_{s^-})$, then the process
\begin{align}\label{J2t}
t \mapsto \sum_{s \leq t,\ s \in \set{T_1,T_2, \cdots, \tau_\partial}} \phi(s^{-},X_{s^-})  -
\int_0^t \phi(u,X_{u}) {\bar{q}(X_u)} \d u
\end{align}
is a martingale. As a consequence, with the notation of Lemma~\ref{lem:pdmp1},
\[
t  \mapsto \sum_{s \leq t,\ s \in \set{T_1,T_2, \cdots, \tau_\partial}} 
 \psi(s,X_{s^-},X_s)  - \int_0^t\int_{x' \in F \cup \set{ \partial }} \psi(u,X_u,x')q(X_u,\d x') \d u
\]
is a martingale.
\end{Lem}
\begin{proof}
Let $l \mapsto N_l$ denote the Poisson process associated to the jump times of the PDMP. Using the change of time defined by $\d L_t = \bar{q}(X_t) \d t$, it is standard to check that 
$$ t \mapsto N_{L_t} - \int_0^t \bar{q}(X_s) \d s$$ is a \cadlag martingale. Next, let us consider the stochastic integral
\[
\int_0^t \phi(s^{-},X_{s^-}) \p{\d N_{L_s} - \bar{q}(X_s) \d s }.
\]
By construction of the stochastic integral, this is a local martingale. Once again, a comparison with a Poisson process with rate $\norm{\bar{q}}_{\infty}$ ensures that the latter is in fact a martingale.\medskip

For the second point, it suffices to add the martingale of Lemma~\ref{lem:pdmp1} and the martingale defined by (\ref{J2t}) with  
$$\phi(u,x)=\int_{x' \in F \cup \set{ \partial }} \psi(u,x,x') \frac{q(x,\d x')}{\bar{q}(x)}.$$

% since the underlying filtration is right-continuous, we can consider the \cadlag version of $t \mapsto J^1_{t}(\phi)$ as defined in Lemma~\ref{lem:pdmp1} above. On the other hand, by independence of the jumps with respect to the jump times:
% \[
% \int_0^t \p{ \ph(s,X_s) - \ph(s,X_{s^-}) - \frac{q(\ph)(s,X_{s^-})}{\bar{q}(X_{s^-})} } \d  N_{L_s}
% \]
% is also a martingale. Combining the two yields the result.
\end{proof}

Now, if we apply Lemma~\ref{lem:pdmp2} with $\psi(s,x,x') = \un_\partial(x')-\un_\partial(x)$, it is readily seen that Assumption~(SK) is satisfied with $\lambda(x):=q(x,\{\partial\})$.
\medskip

Next, let us consider the \cadlag version of the martingale 
$$\M_t(\ph)\eqdef  Q^{T-t}(\ph)(X_t) =\E[\ph(X_T) | X_t]$$
for some bounded $\ph$ with $\ph( \partial) = 0$. We claim that the martingale property in Assumption~(CC) is satisfied for 
$$\Gamma_t(\ph)(x) =   \int_{F \cup \{\partial\}} \p{Q^t(\ph)(x') - Q^t(\ph)(x)}^2 q(x, \d x'),$$ 
which would imply that 
$$\norm{\Gamma_t(\ph)}_\infty \leq 4\norm{\ph}_\infty^2 \norm{\bar{q}}_\infty,$$ 
and ensure that (CC) holds true. Indeed, one has the following It\^o formula:
$$
\p{Q^{T-t}(\ph)(X_t) }^2 = \p{Q^{T}(\ph)(X_0) }^2 + 2 \int_0^t \M_{s^-}(\ph) \d \M_s(\ph) + \b{\M(\ph),\M(\ph)}_t,
$$
where 
$$\b{\M(\ph),\M(\ph)}_t = \sum_{s \leq t,\ s \in \set{T_1,T_2, \cdots, \tau_\partial}} \p{\M_s(\ph) - \M_{s^-}(\ph)}^2$$
since $\M(\ph)$ is deterministic between jump times. Now, considering
$$\psi(s,x,x') =(Q^{T-s}(\ph)(x')-Q^{T-s}(\ph)(x))^2$$
in Lemma~\ref{lem:pdmp2} implies that
$$ t \mapsto  \sum_{s \leq t,\ s \in \set{T_1,T_2, \cdots, \tau_\partial}} \p{\M_s(\ph) - \M_{s^-}(\ph)}^2 - \int_0^t \Gamma_{T-s}(\ph)(X_{s^-}) \d s. $$
is a martingale, which terminates the proof of Proposition \ref{pdmp}.

\subsection{Martingales from individual particles}
Throughout this section, we detail some results on the martingales at stake in the Fleming-Viot particle system. We recall that $(\calF_t)_{t \geq 0}$ denotes the right-continuous filtration generated by the particle system. Let $1 \leq n \leq N$ and $k \geq 0$ be given, and consider the particle with index $n$ and the associated $k$-th branching time $\tau_{n,k}$.  In order to lighten the notation, we skip the indices of the particle and of the branching time, and we denote in the aftermath $\tau=\tau_{n,k}$ as well as $\tau_{\partial}=\tau_{n,k+1}$. As before, we also set $Q^t(\ph)(\partial) = 0$.\medskip

Let us denote $(\tilde{X}_t)_{t\geq 0}$ the particle equal to $X_t$ for $t < \tau_{\partial}$ and which stays in the cemetery point $\partial$ instead of branching for $t \geq \tau_{\partial}$. In the next lemma, we use the Markov property of $\tilde{X}$ with respect to its right-continuous filtration in order to obtain c\`adl\`ag martingales.
\begin{Lem}\label{lem:mart0} For any bounded $\ph$, the process
\begin{equation}\label{eq:M_tech}
 t \mapsto \mathbb{M}_t \equiv \M_t(\ph) \eqdef \un_{t \geq \tau} \b{ Q^{T-t}(\ph)(\tilde{X}_{t}) - Q^{T-\tau}(\ph)(\tilde{X}_{\tau})   }
\end{equation}
is a $(\calF_t)_{t \geq 0}$-martingale with a c\`adl\`ag version.
\end{Lem}
\begin{proof}
First denote $N_t \eqdef   \E[\ph(\tilde{X}_T)| \calF_t]$, which is a martingale with a c\`adl\`ag version. Existence of a c\`adl\`ag version stems from the right continuity of $(\calF_t)_{t \geq 0}$ and from the classical formulas $$\lim_{\eps \downarrow 0} \E[X | \calF_{t+\eps} ] =  \E[X | \calF_{t^+}] \mbox{ and }
\lim_{\eps \downarrow 0} \E[X | \calF_{t-\eps}]=  \E[X | \calF_{t^-} ],$$ where $\eps$ ranges in any countable dense subset of $\R$ (see for example Corollary~$2.4$, Chapter II in~\cite{ry99}).\medskip

By construction of the particle system, we can then write
\begin{align*}
N_t& = %\un_{t \geq \tau} \E \b{\left. \ph(\tilde{X}_T) \right| \calF_t } = 
\E \b{\left. \E \b{\left.  \ph(\tilde{X}_T) \right | \calF_{\tau \vee t} } \right| \calF_t }\\
&= \E \b{\left. \E \left[ \ph(\tilde{X}_T)\left| \calF_\tau \vee \bigcap_{\eps >0} \sigma( \tilde{X}_{s}, \, \tau \leq s \leq t\vee \tau + \eps)\right. \right] \right| \calF_t },
\end{align*}
In the latter, remark that $\tau$ is measurable with respect to $\calF_{\tau }$ so that the (not necessarily strong) Markov property of $\tilde{X}$ can be applied at time $t \vee \tau$  conditionally on $\calF_\tau$ to get
\[
N_t = \E \b{\left. Q^{T-t\vee \tau}(\ph)(\tilde{X}_{t \vee \tau }) \right | \calF_t },
\]
which yields $\un_{t \geq \tau} N_t = \un_{t \geq \tau} Q^{T-t}(\ph)(\tilde{X}_{t})$. Finally, notice that 
$$ \M_t = \un_{t \geq \tau}\p{ N_t - N_{t \wedge \tau} } = N_t - N_{t \wedge \tau}$$
is a c\`adl\`ag martingale by Doob's optional sampling, which proves the result.

%***Since we work witrh right-continuous filtrations, all martingales have a cad-lag version (see~\cite{ry99}). A VOIR***
\end{proof}

\begin{Lem}\label{lem:comp}
For any~$\ph$ satisfying Assumption~(CC), one has
 \[
  \bracket{\mathbb{M},\mathbb{M}}_t = \int_\tau^{t \vee \tau} \Gamma_{T-s}(\ph)(\tilde{X}_s) \d s .
 \]
\end{Lem}
\begin{proof}
Let us introduce the notation
$$L_{t} \eqdef \b{ Q^{T-t \vee \tau}(\ph)(\tilde{X}_{t \vee \tau}) }^2 - \b{ Q^{T-\tau}(\ph)(\tilde{X}_\tau) }^2 - \int_\tau^{t \vee \tau} \Gamma_{T-s}(\ph)(\tilde{X}_s) \d s . $$
Note that $t \mapsto L_t$ is obviously adapted and c\`adl\`ag from Lemma~\ref{lem:mart0}. \medskip

First, an elementary calculation shows that
\begin{align*}
 L_t - \mathbb{M}_t^2 + \int_\tau^{t \vee \tau} \Gamma_{T-s}(\ph)(\tilde{X}_s) \d s &= 2  Q^{T-\tau}(\ph)(\tilde{X}_\tau) \M_t \\
 &= 2  Q^{T-\tau \wedge t}(\ph)(\tilde{X}_{\tau \wedge t}) \p{\M_{t} - \M_{\tau \wedge t}} ,
 \end{align*}
so that $t \mapsto L_t - \mathbb{M}_t^2 + \int_\tau^{t \vee \tau} \Gamma_{T-s}(\ph)(\tilde{X}_s) \d s $ defines an adapted martingale. \medskip

Second, we claim that $t \mapsto L_t$ is also a martingale. Indeed, since $L_t$ is $\sigma( \tilde{X}_s, \tau \leq s \leq t)$-measurable, the construction of the particle system ensures that, for any $h\geq 0$, 
$$\E \b{ L_{t+h} | \calF_{\tau \vee t} } = E \b{ L_{t+h} \left| \calF_{\tau} \bigvee \bigcap_{\eps >0} \sigma( \tilde{X}_{s}, \, \tau \leq s \leq t + \eps)\right. }.$$ 
Since $\tau$ is obviously $\calF_\tau$ measurable, we can then apply the martingale property of Assumption~(CC) conditionally on $\calF_{\tau}$ between times $t \vee \tau$ and $(t+h) \vee \tau$ to deduce that $\E \b{L_{t+h} - L_t | \calF_{\tau \vee t } }=0$, which yields $\E \b{L_{t+h} - L_t | \calF_{ t } }=0$ and proves the claim. \medskip

Thus, the process
$$t \mapsto \mathbb{M}_t^2 - \int_\tau^{t \vee \tau} \Gamma_{T-s}(\ph)(\tilde{X}_s) \d s $$
is a martingale, which gives the desired result by definition of the predictable quadratic variation.

\end{proof}

\begin{Lem}\label{akcin}
Under Assumption~(SK), the process
 \[
  t \mapsto \un_{\tilde{X}_t = \partial} - \int_\tau^{t\wedge \tau_\partial} \lambda(\tilde{X}_s) \, \d s
 \]
 is a c\`adl\`ag $(\calF_t)_{t \geq 0}$-martingale. 
\end{Lem}
\begin{proof}
 As in the proof of Lemma~\ref{lem:mart0}, by construction of the particle system, conditioning $\un_{\tilde{X}_{t+h} = \partial} - \int_\tau^{(t+h)\wedge \tau_\partial} \lambda(\tilde{X}_s) \, \d s$ with respect to $\calF_{\tau \vee t}$ amounts to condition with respect to 
$$\calF_\tau \bigvee \bigcap_{\eps >0} \sigma( \tilde{X}_{s}, \, \tau \leq s \leq t\vee \tau + \eps).$$ Applying Assumption~(SK) then yields the result.
\end{proof}

A classical consequence of Assumptions (SK) and (CC) is the so-called ``quasi-left continuity'' of the martingales we have just constructed (see for example Definition 2.25, Chapter I in \cite{js03} for a definition). For the sake of completeness, we recall a weaker consequence which will be useful for our purpose. 

\begin{Lem}\label{lem:quasi-left_1}
Let $\sigma,S$ be two stopping times. Assume that $\sigma$ and $(\tilde{X}_t)_{t \geq \tau}$ are conditionally independent given $\calF_S$. Then, under Assumption~(SK), 
$$\P( \tau_\partial = \sigma | \calF_{\sigma^-}) \un_{S < \sigma < +\infty} = 0.$$
\end{Lem}
\begin{proof}
From Lemma~\ref{akcin}, we know that the process
\[
  t \mapsto L_t \eqdef   \un_{\tilde{X}_t = \partial} - \int_\tau^{t\wedge \tau_\partial} \lambda(\tilde{X}_s) \, \d s
 \]
is a c\`adl\`ag martingale. By continuity of the integral, the latter martingale makes only one jump of value $+1$ at time $\tau_\partial$. As a consequence, using Lemma~\ref{lem:zero_jump2}, we have that
$\P( \tau_\partial = \sigma  |  \calF_{\sigma^-} ) = 0 $.
%\E ( B^n_{(t+h)\wedge  \tau^{n,k} } - B^n_{t \vee  \tau^{n,k-1}} |  \calF_{\tau^{n,k-1}}  ) 
%\leqslant \delta \norm{\lambda}_\infty.$$
%Letting $\delta \to 0$, we infer that for any $t>0$,
%$\P( \tau_\partial=S+h |  \calF_{S} )  = 0$.
\end{proof}

\begin{Lem}\label{lem:quasi-left_2}
Let $\sigma, S  \in [0, + \infty)$ be two stopping times such that $\sigma$ and $(\tilde{X}_t)_{t \geq \tau}$ are conditionally independent given $\calF_S$. If the function $\ph$ satisfies Assumption~(CC), then 
$$\P( \Delta \mathbb{M}_{\sigma}(\ph) \neq 0| \calF_{\sigma^-}) \un_{S < \sigma < +\infty  } = 0,$$
where $\M_t(\ph)$ is defined by~\eqref{eq:M_tech}.
\end{Lem}
\begin{proof}
Let us recall that for any c\`adl\`ag martingale $\mathbb{M}$, $\Delta \b{\mathbb{M},\mathbb{M}} = \p{\Delta \mathbb{M}}^2$. Therefore, proving that $\P( \mathbb{M}_{\sigma} \neq \mathbb{M}_{\sigma^-} | \calF_S) \un_{S < \sigma < +\infty  } = 0$ is equivalent to show that
$$\P( \b{\mathbb{M},\mathbb{M}}_{\sigma} \neq \b{\mathbb{M},\mathbb{M}}_{\sigma^-} | \calF_S) \un_{S < \sigma < +\infty  } = 0,$$ 
which in turn is equivalent for square integrable martingales to 
$$ \E \b{\Delta \b{\mathbb{M},\mathbb{M}}_{\sigma}  | \calF_S}  \un_{S < \sigma < +\infty  }= 0.$$ 
According to Lemma~\ref{lem:comp}, under Assumption~(CC) the compensator of $\b{\mathbb{M},\mathbb{M}}_t$ is given by
$$t \mapsto \bracket{\M,\M}_t = \int_\tau^{t \wedge \tau_\partial} \Gamma_{T-s}(\ph)(\tilde{X}_s) \d s,$$
and is continuous with respect to $t$.\medskip

Finally, remember that the process $t \mapsto L_t \eqdef   \bracket{\mathbb{M},\mathbb{M}}_t - \b{\mathbb{M},\mathbb{M}}_t $ is a martingale so that, thanks to Lemma~\ref{lem:zero_jump2},
$$\E \b{\Delta \b{\mathbb{M},\mathbb{M}}_{\sigma} | \calF_S} \un_{S < \sigma < +\infty  } = \E \b{\Delta \bracket{\mathbb{M},\mathbb{M}}_{\sigma}      }  \un_{S < \sigma < +\infty  }=0,$$

%\leq \int_{S+h-\delta}^{S+h} \norm{\Gamma \p{Q^{T-s}(\ph),Q^{T-s}(\ph)}}_\infty ds \xrightarrow[\delta \to 0]{} 0,$$
and the result follows.
\end{proof}

\subsection{Stopping times and martingales}

We first start with an intuitive fact that can be easily verified.
\begin{Lem}\label{albcios}
 Let $\tau$ be a stopping time on a filtered probability space, and $U$ an integrable and $\calF_\tau$ measurable random variable such that $\E \b {U | \calF_{\tau-}}=0$. Then the process $t \mapsto U \un_{t \geq  \tau}$ is a c\`adl\`ag martingale. 
\end{Lem}
\begin{proof}
 Let $t > s$ be given. First remark that $\un_{t \geq  \tau} = \un_{s \geq  \tau} +  \un_{s <  \tau} \un_{t \geq  \tau} $. Then by definition of $\calF_\tau$, $U\un_{s \geq  \tau} $ is $\calF_s$-measurable, so that
 \[
 \E \b{ U \un_{t \geq  \tau} | \calF_s} =  U\un_{s \geq  \tau} + \E \b{U \un_{t \geq  \tau} 
  | \calF_s}\un_{s <  \tau}.
  \]
Next, by definition of $\calF_{\tau-}$, $\E \b{U \un_{t \geq  \tau} 
  | \calF_s}\un_{s <  \tau} $ and $\un_{t \geq  \tau}$ are $\calF_{\tau-}$-measurable, so that
  \[
   \E \b{U \un_{t \geq  \tau} 
  | \calF_s}\un_{s <  \tau}  =\E \b{\E \b{ U | \calF_{\tau-} }\un_{t \geq  \tau} 
  | \calF_s}\un_{s <  \tau}  =0.
   \]
The result follows.
\end{proof}

\begin{Lem}\label{lem:zero_jump2}
Let $S,\sigma $ be two stopping times, and let $L$ be a c\`adl\`ag bounded martingale. 
Assume that $\sigma$ and $\p{L_t}_{ t \geq 0 }$ are independent conditionally on $\calF_S$. 
Then $\E \b{\Delta L_\sigma|\calF_{\sigma^-}}\un_{S<\sigma<+\infty} =0$.
\end{Lem}
\begin{proof}
{\it Step~$1$: Case $S=0$.} Let us prove the result in the special case $S=0$. 
Let us define the enlarged filtration $(\calG_t)_{t\geq 0}$ as the smallest filtration containing $\calF_t$ for each $t\geq 0$ and such that $\sigma$ is $\calG_0$ measurable. The independence of $\p{L_t}_{ t \geq 0 }$ and $\sigma$ 
implies that $L_t$ is still a $\calG_t$-martingale, and that for any $n\ge 1$, 
\begin{align*}
\sigma_n=(1-1/n)\sigma1_{\sigma<+\infty}+n1_{\sigma=+\infty}
\end{align*}
is a stopping time. Since $L$ is bounded and thus uniformly integrable, we can apply Doob's optional sampling theorem:
\begin{align*}
\E \b{L_{\sigma}|\calG_{\sigma_n} }= L_{\sigma_n}.
\end{align*}
Notice that $\bigvee_n\calG_{\sigma_n}=\calG_{\sigma^-}$ since $\sup_n \sigma_n = \sigma$ with $\sigma_n < \sigma $ on the event $\set{\sigma > 0}$.
% because any basic set of $\calG_{\sigma^-}$ 
% with the form $B\cap \{t<\sigma\}$, $B\in\calG_t$ can be rewritten as
% \begin{align*}
% B\cap \{t<\sigma\}=\bigcup_n \Big( B\cap \{t<\sigma_n\}\Big).
% \end{align*}
% and each set in the union belongs to $\calG_{\sigma_n}$. 
% $\sigma=+\infty$ NE POSE PAS PB MAIS CONCERNANT $\sigma=0$ JE NE SAIS PAS TROP 
% COMMENT ON FAIT AVEC $\calF_{\sigma^-}$...
One gets with the martingale convergence theorem in the limit $n \to +\infty$:
\begin{align*}
\E \b{L_{\sigma}|\calG_{\sigma^-}} \un_{\sigma > 0}= L_{\sigma^{-}}\un_{\sigma > 0}.
\end{align*}

{\it Step~$2$: Any $S$.} For any $t\ge 0$, $t\vee S$ is a stopping time, thus, 
by Doob's optional sampling theorem, $M_t \eqdef L_{t \vee S}$ is a 
$\calG_t \eqdef\calF_{t \vee S}$-martingale. Applying Step~1 to $M,\calG$ 
($\sigma$ and $\p{M_t}_{ t \geq 0 }$ are indeed independent given $\calF_S$), we get
\begin{align*}
\E \b{L_{\sigma\vee S}|\calF_{(\sigma\vee S)^-} }\un_{\sigma > 0}= L_{(\sigma\vee S)^{-}}\un_{\sigma > 0}.
\end{align*}
Since $\set{\sigma > S}\in\calF_{\sigma^-}$ (cf.~Section\,\ref{remind}), 
multiplying both sides by $\un_{\sigma > S}$
leads to
\begin{align*}
\E \b{L_\sigma|\calF_{(\sigma\vee S)^-} }\un_{\sigma > S}= L_{\sigma^{-}}\un_{\sigma > S}.
\end{align*}
It remains to check that for any bounded variable $X$, one has 
\begin{align*}
\E(X | \calF_{(\sigma \vee S)^-} ) \un_{S < \sigma} =
\E(X | \calF_{\sigma^-}) \un_{S < \sigma}.
\end{align*}
For this, it suffices to prove that the l.h.s.~belongs to $\calF_{\sigma^-}$, or more generally that 
for any $A\in\calF_{(\sigma \vee S)^-} $ the set $A \cap \set{\sigma > S}$ belongs to $\calF_{\sigma^-}$
(we have already pointed out that $\set{\sigma > S}\in\calF_{\sigma^-}$).
Let us check this for the generating sets $A=B\cap \{t<\sigma\vee S\}$, with $B\in\calF_t$. Such a set
can be rewritten as
\begin{align*}
B\cap \{t<\sigma\vee S\}\cap \{\sigma> S\}=
\Big(B\cap \{t<\sigma\}\Big)\cap \{\sigma> S\}
\end{align*}
which is the intersection of two sets of $\calF_{\sigma^-}$, hence the desired result.
\end{proof}

\bibliographystyle{plain}
\bibliography{biblio-cdgr}
\end{document}